\documentclass[12pt]{amsart}

\usepackage{amssymb, amscd, txfonts}
\usepackage{graphicx}
\usepackage[all]{xy} 

\numberwithin{equation}{section}

\newtheorem{theorem}{Theorem}[section]
\newtheorem{proposition}[theorem]{Proposition}
\newtheorem{lemma}[theorem]{Lemma}

\theoremstyle{definition}
\newtheorem{definition}[theorem]{Definition}
\newtheorem{example}[theorem]{Example}

\theoremstyle{remark}
\newtheorem{remark}[theorem]{Remark}

\newcommand{\Z}{\mathbb{Z}}
\newcommand{\Q}{\mathbb{Q}}
\newcommand{\R}{\mathbb{R}}
\newcommand{\C}{\mathbb{C}}

\newcommand{\proj}{{\mathbb P}}
\newcommand{\cube}{\square}

\newcommand{\D}{\mathcal{D}}
\newcommand{\G}{\Gamma}
\newcommand{\LL}{\mathcal{L}}
\newcommand{\E}{\mathcal{E}}
\newcommand{\El}{\mathcal{E}_{\lambda}}
\newcommand{\Sym}{{\rm Sym}}
\newcommand{\prim}{{\rm prim}}
\newcommand{\AG}{\mathcal{A}_{\Gamma}}
\newcommand{\AGast}{\mathcal{A}_{\Gamma}^{\ast}}
\newcommand{\AGtor}{\mathcal{A}_{\Gamma}'}
\newcommand{\AI}{\mathcal{A}_{I}}
\newcommand{\AItilde}{\widetilde{\mathcal{A}}_{I}}
\newcommand{\VI}{\mathcal{V}_{I}}
\newcommand{\XG}{\mathcal{X}_{\Gamma}}
\newcommand{\XGtor}{\mathcal{X}_{\Gamma}'}
\newcommand{\St}{{\rm St}}

\begin{document}

\title[]{Siegel modular forms arising from higher Chow cycles}
\author[]{Shouhei Ma}
\thanks{Supported by KAKENHI 21H00971} 
\address{Department~of~Mathematics, Science~Institute~of~Tokyo, Tokyo 152-8551, Japan}
\email{ma@math.titech.ac.jp}

\begin{abstract}
We prove that the infinitesimal invariant of a higher Chow cycle of type $(2, 3-g)$ on 
a generic abelian variety of dimension $g\leq 3$ gives rise to a meromorphic Siegel modular form 
of (virtual) weight ${\Sym}^4\otimes \det^{-1}$ with bounded singularity, 
and that this construction is functorial with respect to rank $1$ degeneration, 
namely the K-theory elevator for the cycle corresponds to the Siegel operator for the modular form. 
\end{abstract}

\maketitle

\section{Introduction}\label{sec: intro}

The theory of modular forms and that of algebraic cycles 
have a common origin in the study of Gauss, Abel and Jacobi on the Lemniscate. 
Both theories have developed vastly in this two hundred years, 
and accordingly, they have gradually separated, as is usual with the history of Mathematics. 
In the modular form side, one direction of development was 
\begin{equation*}
\textrm{one variable} \: \rightsquigarrow \: \textrm{several variables} \: \rightsquigarrow \: \textrm{vector-valued}
\end{equation*}  
The most well-understood case in this direction is vector-valued Siegel modular forms;  
see \cite{vdG} for a survey. 
On the other hand, in the cycle side, the object of study has grown as  
\begin{equation*}
\textrm{divisors} \: \rightsquigarrow \: \textrm{algebraic cycles} \: \rightsquigarrow \: \textrm{higher Chow cycles} 
\end{equation*}  
The last one, higher Chow cycles, were introduced by Bloch \cite{Bl1} 
as the cycle-theoretic incarnation of higher algebraic K-theory. 

Our purpose in this paper is to show that even at this stage, the two branches are still connected in low dimension. 
In order to ease the exposition, let us work in an (over)simplified setting, 
referring the reader to later sections for the full version. 
Let ${\AG}={\D}/{\G}$ be a Siegel modular variety of genus $g\leq 3$, where the arithmetic group ${\G}$ does not contain $-1$. 
Let $\pi\colon {\XG}\to {\AG}$ be the universal family of abelian varieties over ${\AG}$. 
Our starting point is the following observation. 

\begin{theorem}[Theorem \ref{thm: main construction}]\label{thm: intro main construction}
Let $U$ be a Zariski open set of ${\AG}$ and 
$Z$ be a family of higher Chow cycles of type $(2, 3-g)$ on $\pi^{-1}(U)\to U$,  
nullhomologous when $g=3$. 
Then the primitive infinitesimal invariant of $Z$ 
gives a meromorphic Siegel modular form $f_{Z}$ of weight ${\Sym}^4\otimes \det^{-1}$ 
which is holomorphic over $U$ and has at most pole of order $1$ along the complement of $U$. 
When $g=1$, the weight $3$ modular form $f_{Z}$ is holomorphic at the cusps. 
\end{theorem}

We say that this setting is oversimplified because 
usually a family of (higher) cycles can be defined only after taking an etale base change of 
the given family of varieties. 
Then we need to take symmetric polynomials along the fibers of the base change for obtaining Siegel modular forms. 
This produces Siegel modular forms of weight ${\Sym}^{4i}\otimes \det^{-i}$ for $i$ up to the degree of the base change. 
The full version (Theorem \ref{thm: main construction}) is formulated in this setting. 

In the case $g=1$, if the $(2, 2)$-cycles are defined on every elliptic curve, 
we obtain a holomorphic elliptic modular form of weight $3$. 
This weight is conveniently small but not very small. 
On the other hand, when $g\geq 2$, the basic theory of Siegel modular forms tells us that 
there is no holomorphic modular form of weight ${\Sym}^4\otimes \det^{-1}$. 
Thus, if $f_{Z}\not\equiv 0$, it must have a pole. 
In other words, if $Z$ has nontrivial primitive infinitesimal invariant, it cannot be extended over a divisor of ${\AG}$. 
This is an easy implication from modular forms to algebraic cycles. 

One might think, like myself when started this project, that
the same construction is possible for other types of cycles and arbitrary $g$ 
to produce plenty Siegel modular forms. 
Indeed, we have the infinitesimal invariant for any type of cycles. 
However, it turns out from the calculation of Nori (\cite{No} p.372) 
that the Koszul cohomology sheaf where the primitive infinitesimal invariant lives 
is nonvanishing only in the above cases, 
beyond the classical case of $0$-cycles (and divisors). 
I now understand that this is a sort of exceptional phenomenon in low dimension, 
like the surjectivity of the Torelli map in $g \leq 3$. 
Moreover, by some rigidity theorems in the theory of algebraic cycles, 
we find that only countably many modular forms can be obtained in this way (\S \ref{ssec: rigidity}). 
Thus these modular forms should be precious. 

Theorem \ref{thm: intro main construction} is not surprising 
if we 
view the Koszul complex as a complex of automorphic vector bundles. 
The boundary behavior of $f_{Z}$ is a consequence of 
the admissibility of the normal function of $Z$ (\cite{GG}, \cite{Sa}). 
Nevertheless, in view of the current distance of the two branches, 
it would not be entirely trivial to make this link explicit, 
which will serve as the basis of further investigation. 
In the case $g=3$, a similar observation was also made recently by Hain \cite{Hain3} 
in the context of Teichm\"uller modular forms. 
In a more primitive form, this type of observation was done for $K3$ surfaces in \cite{Ma1} \S 3.8. 

Our more practical effort in this paper is devoted to showing that 
the correspondence $Z\rightsquigarrow f_{Z}$ is functorial with respect to degeneration. 
Let $g\geq 2$. 
We choose a corank $1$ (i.e., maximal) cusp ${\AI}$ of ${\AG}$, 
which itself is a Siegel modular variety of genus $g-1$. 
In the modular form side, we have the \textit{Siegel operator} $\Phi_{I}$ 
which is a sort of restriction operator to ${\AI}$ (see \cite{We}, \cite{vdG}). 
This produces vector-valued Siegel modular forms on ${\AI}$ from those on ${\AG}$. 
We extend this operation to meromorphic modular forms (\S \ref{sec: Siegel operator}). 
On the other hand, we have the rank $1$ degeneration of the universal abelian variety 
over the toroidal boundary divisor over ${\AI}$. 
For a cycle-theoretic technical reason, 
we restrict ourselves to the case of irreducible degeneration (cf.~Remark \ref{remark: elevator with several components}). 
Then, for a cycle family $Z$ as above, 
we have a degeneration operation known as the \textit{K-theory elevator} (\cite{Co}, \cite{elevator}). 
In an (over)simplified setting, this produces a family $Z_{I}$ of higher Chow cycles on 
the $(g-1)$-dimensional universal family over ${\AI}$ 
whose K-theoretic degree increases by $1$, 
namely if $Z$ has type $(2, n)$ then $Z_{I}$ has type $(2, n+1)$. 

Our main result is that the K-theory elevator corresponds to the Siegel operator: 

\begin{theorem}[Theorem \ref{thm: elevator Siegel}]\label{thm: degeneration intro}
We have $\Phi_{I}f_{Z}=f_{Z_{I}}$ up to constant. 
\end{theorem}

See \S \ref{sec: elevator} for the precise setting of Theorem \ref{thm: degeneration intro}. 
The full version (Theorem \ref{thm: elevator Siegel}) is again formulated for base-changed families. 

The proof of Theorem \ref{thm: elevator Siegel} is done by comparing the Siegel operator with 
the limit formula of \cite{elevator} for the normal functions. 
The partial toroidal compactification serves as a common place for both operations. 
A point of the proof is that the reduction of automorphic vector bundles which takes place in the process of Siegel operator 
agrees with the assumption made in \cite{elevator} on the Kunneth type of differential forms. 

When two branches intersect, this often occurs on the ``boundary'' of them; 
one is then invited to extend the scope of each side. 
In a modest and technical level, this seems to be the case also in the subject of this paper. 
The appearance of \textit{meromorphic} modular forms is already such a sign. 
In the direction of Theorem \ref{thm: degeneration intro}, we are naturally led to considering 
\begin{enumerate}
\item Siegel operator for meromorphic modular forms (\S \ref{ssec: Siegel meromorphic}) 
\item Structure of rank $1$ degeneration of universal abelian variety 
(see \cite{Na}, \cite{HW}, \cite{HKW} for some examples) 
\item K-theory elevator for chains of ${\proj}^1$-bundles (cf.~Remark \ref{remark: elevator with several components}) 
\end{enumerate}
(1) and (2) belong to the modular form side, while (3) belongs to the cycle side. 
Here we just prepare (1) in a minimal level, which is necessary for presenting Theorem \ref{thm: degeneration intro}. 
(2) and (3) are related, and will be necessary when extending Theorem \ref{thm: degeneration intro}. 
This will be addressed elsewhere. 

Beyond these rather technical problems, Theorem \ref{thm: degeneration intro} also raises the following question. 
From the viewpoint of modular forms, Siegel operator just singles out the $0$-th Fourier-Jacobi coefficient, 
while we also have Fourier-Jacobi coefficients of higher degree. 
They are no longer Siegel modular forms, but Jacobi forms. 
Then one might want to expect that 
at least the first Fourier-Jacobi coefficient of $f_{Z}$ has a cycle-theoretic counterpart,  
which could be interpreted as an operation of differentiating degenerating family of cycles. 

If we go back to Theorem \ref{thm: intro main construction}, it also raises the following problems: 
\begin{itemize}
\item Arithmetic properties of $f_{Z}$ when $Z$ is defined over ${\Q}$. 
\item Explicit expression of $f_{Z}$ when $Z$ is concretely given; 
for example, when $g=2$ and the exceptional locus is Humbert surfaces, 
then express $f_{Z}$ as the product of a Borcherds product and an explicit holomorphic vector-valued Siegel modular form. 
\item In the above situation, we may take the residue of $f_Z$ along the Humbert surface 
to obtain a vector-valued modular form on the Humbert surface. 
Does this operation have a cycle-theoretic interpretation? 
\end{itemize}

The motivation of this paper comes from the pioneering work of Collino and Pirolla (\cite{Co}, \cite{CP}). 
In \cite{CP}, they discovered that the infinitesimal invariant of the Ceresa cycle in genus $3$ 
is essentially the equation of the universal plane quartic. 
Much later, the latter was studied independently from the viewpoint of modular forms (\cite{CFvdG2}); 
the connection was recently clarified in \cite{Hain3}.  
In the subsequent paper \cite{Co}, 
Collino constructed higher Chow cycles in $g=2$ and $g=1$ 
by applying the K-theory elevator to the Ceresa cycle in $g=3$ successively, 
and expressed their infinitesimal invariants in terms of 
invariants of six points on ${\proj}^1$ in the case $g=2$, 
and theta series in the case $g=1$ respectively. 
It was observed, again later and independently, 
that invariants of binary sextics are vector-valued Siegel modular forms (\cite{CFvdG1}). 
In a sense, like the age of Lemniscate, algebraic cycles have been ahead of modular forms. 
The present paper grew out of an effort to understand the idea of Collino from the viewpoint of modular forms. 


I would like to thank Richard Hain for fruitful communication.

\section{Siegel modular forms}\label{sec: SMF}

In this section we recall Siegel modular varieties and Siegel modular forms following \cite{vdG}.

\subsection{Siegel modular varieties}\label{ssec: SMV}

Let $\Lambda$ be a symplectic lattice, namely 
a free ${\Z}$-module of rank $2g>0$ equipped with a nondegenerate symplectic form 
$( \cdot, \cdot) \colon \Lambda \times \Lambda \to {\Z}$. 
We denote by ${\rm Sp}(\Lambda)$ the symplectic group of $\Lambda$. 
Let ${\rm LG}(\Lambda)$ be the Lagrangian Grassmannian parametrizing 
maximal ($\Leftrightarrow$ $g$-dimensional) isotropic subspaces $V$ of $\Lambda_{{\C}}$. 
Let ${\D}\subset {\rm LG}(\Lambda)$ be the open set parametrizing those $V$ such that 
the Hermitian form $i( \cdot, \bar{\cdot})|_{V}$ on $V$ is positive-definite. 
This is the Hermitian symmetric domain associated to ${\rm Sp}(\Lambda_{{\R}})$. 

We choose a maximal rational isotropic subspace $I_{0}$ of $\Lambda_{{\Q}}$. 
Let 
\begin{equation*}
\frak{H}_{I_{0}} = \{ \: \tau \in {\Sym}^2(I_0)_{{\C}} \: | \: {\rm Im}(\tau) > 0 \: \} 
\end{equation*}
be the Siegel upper half space in ${\Sym}^2(I_0)_{{\C}}$, 
where $>0$ means positive-definite. 
If we choose another maximal rational isotropic subspace of $\Lambda_{{\Q}}$ complementary to $I_{0}$, 
the graph construction gives an isomorphism 
\begin{equation}\label{eqn: Siegel upper half space}
{\D} \stackrel{\simeq}{\longrightarrow} \frak{H}_{I_{0}} \; \subset {\Sym}^2(I_0)_{{\C}}. 
\end{equation}
This is a well-known realization of ${\D}$ as the Siegel upper half space. 
Let $P(I_{0})$ be the stabilizer of $I_{0}$ in ${\rm Sp}(\Lambda_{{\Q}})$ 
and $U(I_{0})$ be the kernel of the natural map 
$P(I_{0}) \to {\rm GL}(I_{0})$. 
Then $U(I_0)$ is the unipotent radical of $P(I_{0})$ and naturally isomorphic to ${\rm Sym}^2I_0$. 
The action of $U(I_{0})$ on ${\D}$ is identified via \eqref{eqn: Siegel upper half space} with 
the translation by ${\Sym}^2I_0$ on $\frak{H}_{I_{0}}\subset {\Sym}^2(I_0)_{{\C}}$. 

In general, a rational isotropic subspace $I$ of $\Lambda_{{\Q}}$ determines 
a rational boundary component ${\D}_{I}$ of ${\D}$, usually called a \textit{cusp}. 
This is the locus of those $[V]\in {\rm LG}(\Lambda)$ such that 
$I\subset V$ and $i( \cdot, \bar{\cdot})|_{V}$ is positive-semidefinite with kernel $I_{{\C}}$. 
The cusp ${\D}_{I}$ itself is the Hermitian symmetric domain attached to the symplectic space $I^{\perp}/I$. 
We denote by ${\D}^{\ast}$ the union of ${\D}$ and all cusps, equipped with the so-called Satake topology. 

Let ${\G}$ be a finite-index subgroup of ${\rm Sp}(\Lambda)$. 
We assume $-1\not\in {\G}$ throughout this paper. 
By the theory of Satake-Baily-Borel, 
the quotient ${\AGast}={\D}^{\ast}/{\G}$ has the structure of a normal projective variety 
and contains ${\AG}={\D}/{\G}$ as a Zariski open set. 
In particular, ${\AG}$ is a quasi-projective variety, known as a \textit{Siegel modular variety}. 
For each cusp ${\D}_{I}$ of ${\D}$, its image in ${\AGast}$ is denoted by ${\AI}$. 
This itself is a Siegel modular variety attached to $I^{\perp}/I$. 
We use the terminology ``cusp'' also for ${\AI}$. 

\subsection{Automorphic vector bundles}\label{ssec: automorphic VB}

Let ${\E}$ be the weight $1$ Hodge bundle on ${\D}$, 
namely the sub vector bundle of $\Lambda_{{\C}}\otimes \mathcal{O}_{{\D}}$ 
whose fiber over $[V]\in {\D}$ is $V\subset \Lambda_{{\C}}$ itself. 
Let $\lambda=(\lambda_1\geq \cdots \geq \lambda_{g} \geq 0)$ be a non-increasing sequence of nonnegative integers of length $g$.  
This corresponds to the highest weight of an irreducible polynomial representation 
$V_{\lambda}$ of ${\rm GL}_{g}({\C})$. 
Explicitly, $V_{\lambda}$ can be constructed by applying the Schur functor to the standard representation ${\rm St}={\C}^{g}$. 
We write $\lambda\geq 0$ for the condition $\lambda_{g}\geq 0$. 
By applying the Schur functor to ${\E}$, 
we obtain a ${\rm Sp}(\Lambda_{{\R}})$-equivariant vector bundle ${\El}$ on ${\D}$. 
We call ${\El}$ the \textit{automorphic vector bundle} of weight $\lambda$. 

\begin{example}
(1) The weight $\lambda=(1^{k}, 0^{g-k})$ corresponds to the $k$-th exterior tensor $\wedge^{k}{\rm St}$. 
We have ${\El}=\wedge^k{\E}$ in this case. 
When $\lambda=(1^g)$, we especially write $\lambda=\det$ and $\det {\E}={\LL}$. 

(2) The weight $(k, 0^{g-1})$ corresponds to the $k$-th symmetric tensor 
${\Sym}^{k}{\rm St}$. 
We have ${\El}={\Sym}^k{\E}$ in this case. 
The cotangent bundle $\Omega_{{\D}}^1$ is isomorphic to ${\Sym}^2{\E}$. 
\end{example}

More generally, if we drop the condition $\lambda_{g} \geq 0$, 
a non-increasing sequence $\lambda=(\lambda_1, \cdots , \lambda_{g})$ of integers of length $g$ 
corresponds to the highest weight of an irreducible rational representation $V_{\lambda}$ of ${\rm GL}(g, {\C})$. 
If we write $\lambda=\lambda^{+}-(k, \cdots, k)$ with $\lambda^{+}\geq 0$, 
we have $V_{\lambda}=V_{\lambda^{+}}\otimes \det^{-k}$. 
We often write $\lambda=\lambda^{+}\otimes \det^{-k}$ instead of $\lambda=\lambda^{+}-(k, \cdots, k)$. 
For automorphic vector bundles, we set 
${\El}={\E}_{\lambda^{+}}\otimes {\LL}^{\otimes -k}$ 
accordingly. 

For a maximal isotropic subspace $I_0$ of $\Lambda_{{\Q}}$, 
the symplectic pairing of $V\subset \Lambda_{{\C}}$ with $I_0$ 
defines an isomorphism ${\E}\to (I_{0}^{\vee})_{{\C}}\otimes \mathcal{O}_{{\D}}$. 
For a weight $\lambda\geq 0$, we denote by $V(I_0)_{{\lambda}}$ 
the linear space obtained by applying the Schur functor to $(I_{0}^{\vee})_{{\C}}$. 
Then we obtain an isomorphism 
${\El}\to V(I_0)_{{\lambda}}\otimes \mathcal{O}_{{\D}}$, 
which we call the \textit{$I_{0}$-trivialization} of ${\El}$. 
For general $\lambda$, we write $\lambda=\lambda^{+}\otimes \det^{-k}$ as above and put 
$V(I_0)_{{\lambda}}=V(I_0)_{{\lambda^{+}}}\otimes (\det I_{0})_{{\C}}^{\otimes k}$. 
Then we obtain a trivialization 
${\El}\to V(I_0)_{{\lambda}}\otimes \mathcal{O}_{{\D}}$ similarly.

\subsection{Siegel modular forms}\label{ssec: SMF}

Let $\lambda=(\lambda_1, \cdots , \lambda_{g})$ be a nontrivial highest weight for ${\rm GL}(g, {\C})$. 
When $g\geq 2$, a ${\G}$-invariant holomorphic section $f$ of ${\El}$ over ${\D}$ is called 
a holomorphic \textit{Siegel modular form} of weight $\lambda$ with respect to ${\G}$. 
A basic vanishing theorem says that such a section $f\not\equiv 0$ exists 
only when $\lambda_{g}>0$ (see \cite{vdG} Proposition 1). 
If $\lambda=\det^k$, $f$ is said to be scalar-valued of weight $k$. 
When ${\G}$ is torsion-free, ${\El}$ descends to a vector bundle on ${\AG}={\D}/{\G}$, 
again denoted by ${\El}$. 
Then holomorphic Siegel modular forms of weight $\lambda$ are 
the same as holomorphic sections of ${\El}$ over ${\AG}$. 
In the case $g=1$, modular forms are defined similarly except that 
holomorphicity at the cusps are required as usual. 
The cusp condition is automatically satisfied when $g\geq 2$ (the so-called Koecher principle). 

Let $I_{0}$ be a maximal isotropic subspace of ${\Lambda}_{{\Q}}$. 
Via the isomorphism \eqref{eqn: Siegel upper half space} and the $I_{0}$-trivialization of ${\El}$, 
we can regard $f$ as a $V(I_0)_{\lambda}$-valued holomorphic function on $\frak{H}_{I_{0}}$. 
This is invariant under translation by the lattice $U(I_{0})_{{\Z}}=U(I_0)\cap {\G}$ in ${\Sym}^2I_{0}$. 
Hence it admits a Fourier expansion of the form 
\begin{equation}\label{eqn: Fourier expansion}
f(\tau ) = \sum_{l\in U(I_{0})_{{\Z}}^{\vee}} a(l) q^l, 
\qquad a(l)\in V(I_0)_{\lambda}, \quad \tau \in \frak{H}_{I_0}, 
\end{equation}
where $q^l=\exp(2\pi i (l, \tau))$ 
and $U(I_{0})_{{\Z}}^{\vee}\subset {\Sym}^2I_{0}^{\vee}$ is the dual lattice of $U(I_{0})_{{\Z}}$. 
The holomorphicity at the $I_{0}$-cusp says that 
$a(l)\ne 0$ only when $l$ is contained in the cone 
$\mathcal{C}(I_{0})\subset {\Sym}^2(I_{0}^{\vee})_{{\R}}$ 
of positive-semidefinite forms. 

In this paper we will be mainly concerned with \textit{meromorphic} modular forms. 
This is defined as a ${\G}$-invariant meromorphic section $f$ of ${\El}$ over ${\D}$ 
(plus the cusp condition when $g=1$). 
Here $\lambda$ is a highest weight for ${\rm GL}(g, {\C})$, not necessarily $\lambda \geq 0$. 

\begin{lemma}\label{lem: mero SMF}
A meromorphic Siegel modular form of weight $\lambda$ can be written as $f=f_1/f_0$ 
where $f_1$ is a holomorphic Siegel modular form of weight $\lambda^{+}=\lambda\otimes \det^k$ for some $k$ 
and $f_0$ is a scalar-valued holomorphic Siegel modular form of weight $k$. 
\end{lemma}

\begin{proof}
This is standard, but we supply a proof in the case $g\geq 2$ due to lack of explicit reference. 
We first assume ${\G}$ neat. 
Let $j\colon {\AG}\hookrightarrow {\AGast}$ be the inclusion map. 
Since ${\AGast}$ is normal and the boundary ${\AGast}-{\AG}$ has codimension $>1$, 
the direct image $j_{\ast}{\El}$ is a (reflexive) coherent sheaf on ${\AGast}$. 
By the Levi extension theorem, $f$ extends to a meromorphic section of $j_{\ast}{\El}$ over ${\AGast}$. 
By GAGA, it is a rational section. 
Since ${\LL}$ extends to an ample line bundle on ${\AGast}$ by the Baily-Borel theory (again denoted by ${\LL}$), 
there exists a holomorphic section $f_0$ of ${\LL}^{\otimes k}$ over ${\AGast}$ for some $k\gg 0$ such that 
$f_0\otimes f$ is a holomorphic section of $j_{\ast}{\El}\otimes {\LL}^{\otimes k}$. 
Thus $f_0\otimes f$ is a holomorphic Siegel modular form of weight $\lambda\otimes \det^{k}$. 

For general ${\G}$, we choose a neat normal subgroup ${\G}'\lhd {\G}$ of finite index 
and run the above argument on $\mathcal{A}_{{\G}'}$, choosing $f_0$ to be ${\G}/{\G}'$-invariant 
(e.g., by taking the average product).  
\end{proof}

In some literatures, the conclusion of Lemma \ref{lem: mero SMF} is adopted as 
the definition of meromorphic Siegel modular forms. 
For our purpose, it is more convenient to define them just as meromorphic sections. 
By Lemma \ref{lem: mero SMF}, 
a meromorphic Siegel modular form has a Fourier expansion 
$f=\sum_{l}a(l)q^l$ of the same form as \eqref{eqn: Fourier expansion}, 
where now the vectors $l$ range over the intersection of $U(I_0)_{{\Z}}^{\vee}$ with a translation of $\mathcal{C}(I_{0})$.

\section{Higher Chow cycles}\label{sec: Chow}

In this section we recall higher Chow cycles and normal functions. 
Our main references are \cite{KLM} and \cite{GG}. 
Note that we would not need to define higher Chow \textit{groups}, 
and hence avoid technicalities regarding the moving lemma in our presentation. 
(The exception is \S \ref{ssec: decomposable} and \S \ref{ssec: rigidity}, 
but there higher Chow groups play only auxiliary roles.) 
We use the cubical presentation of higher Chow cycles.

\subsection{Higher Chow cycles}\label{ssec: Chow}

Let $X$ be an equidimensional smooth quasi-projective variety. 
Let $p \geq 0$ and $0\leq n \leq 2$. (This covers the cases we are interested in.) 
We write ${\cube}^n=({\proj}^1-\{ 1 \})^n$ for the algebraic $n$-cube with coordinates $(z_1, \cdots, z_n)$. 
The divisors of $X\times {\cube}^n$ defined by $z_i=0, \infty$ are called the codimension $1$ faces of $X\times {\cube}^n$. 
They are naturally identified with $X\times {\cube}^{n-1}$. 

By a \textit{higher Chow cycle} of type $(p, n)$ on $X$, 
we mean a codimension $p$ cycle $Z$ on $X\times {\cube}^n$ which meets 
every codimension $1$ face of $X\times {\cube}^n$ properly and satisfies the cocycle condition that 
$\sum_{i}(-1)^i(Z|_{(z_i=0)}-Z|_{(z_i=\infty)})$ is a degenerate cycle on $X\times {\cube}^{n-1}$, 
i.e., the sum of cycles pulled back from the codimension $1$ faces of $X\times {\cube}^{n-1}$.  
For example, when $n=0$, $Z$ is just a usual cycle on $X$; 
when $n=1$, $Z$ is a cycle on $X\times {\cube}$ satisfying $Z|_{X\times (0) } = Z|_{X\times (\infty) }$. 

Suppose that $X$ is irreducible and projective, 
and let $H^{2p-n}_{\mathcal{D}}(X, {\Z}(p))$ be the Deligne cohomology of $X$. 
This sits in the exact sequence 
\begin{equation*}
0 \to J^{p,n}(X) \to H^{2p-n}_{\mathcal{D}}(X, {\Z}(p)) \to F^{p}H^{2p-n}(X, {\Z}) \to 0, 
\end{equation*}
where 
\begin{equation*}
F^{p}H^{2p-n}(X, {\Z}) = H^{2p-n}(X, {\Z}) \cap F^pH^{2p-n}(X, {\C}) 
\end{equation*}
and $J^{p,n}(X)$ is the generalized complex torus defined by 
\begin{equation*}
J^{p,n}(X) = \frac{H^{2p-n-1}(X, {\C})}{F^{p}H^{2p-n-1}(X, {\C})+H^{2p-n-1}(X, {\Z})}. 
\end{equation*}
(We ignore the constants in the Tate twists.) 
The torus $J^{p,n}(X)$ is called the \textit{generalized Jacobian} of $X$. 
Note that $F^{p}H^{2p-n}(X, {\Z})$ coincides with the torsion part of $H^{2p-n}(X, {\Z})$ when $n>0$. 

According to Bloch \cite{Bl2}, the \textit{regulator} $\nu(Z)$ of 
a higher Chow cycle $Z$ of type $(p, n)$ is defined as an element of $H^{2p-n}_{\mathcal{D}}(X, {\Z}(p))$. 
When $\nu(Z)$ is sent to $0$ in $F^{p}H^{2p-n}(X, {\Z})$, the cycle $Z$ is said to be \textit{nullhomologous}. 
In particular, when $n>0$ and $H^{2p-n}(X, {\Z})$ is torsion-free (e.g., $X$ an abelian variety), 
then any higher cycle $Z$ is nullhomologous. 

For nullhomologous $Z$, the regulator $\nu(Z)\in J^{p,n}(X)$ is also called the \textit{Abel-Jacobi invariant} of $Z$. 
This is a generalization of the classical Abel-Jacobi invariant in the case $n=0$. 
An integral formula in the case $n=1$ has been classically known (going back to Bloch); 
a generalization to the case $n\geq 2$ was given by Kerr, Lewis and M\"uller-Stach \cite{KLM}. 
In this formula, a further proper intersection condition is imposed on $Z$ with regard to the real faces of ${\cube}^{n}$ 
(see \cite{KLM} \S 5.4). 
In \S \ref{sec: elevator} and later, where we implicitly use this formula, 
we tacitly assume that our cycles satisfy this condition.

\subsection{Normal functions}\label{ssec: NF}

Let $\pi \colon X\to S$ be a smooth projective morphism 
between smooth equidimensional quasi-projective varieties with irreducible fibers. 
We write $X_{s}=\pi^{-1}(s)$ for $s\in S$. 
By a family of higher Chow cycles of type $(p, n)$ on $X\to S$, 
we mean a higher Chow cycle $Z$ of type $(p, n)$ on the total space $X$ which meets every fiber 
$X_s\times {\cube}^{n}$ properly and the restriction 
$Z_{s}=Z|_{X_{s}\times {\cube}^n}$ satisfies the proper intersection condition in \S \ref{ssec: Chow} for every $s\in S$. 
Under this requirement, the naive restriction $Z_{s}$ defines a higher Chow cycle on $X_{s}$ 
without resorting to the moving lemma. 

Let $\mathcal{H}^{k}_{{\Z}}$ be the local system $R^{k}\pi_{\ast}{\Z}$ on $S$ 
and $\mathcal{H}^{k}=\mathcal{H}^{k}_{{\Z}}\otimes \mathcal{O}_{S}$.  
We denote by $(\mathcal{F}^{\bullet})$ the Hodge filtration on $\mathcal{H}^{k}$, 
and $\mathcal{H}^{p,q}=\mathcal{F}^{p}/\mathcal{F}^{p+1}$ the Hodge bundles where $p+q=k$.  
The generalized Jacobians of the $\pi$-fibers form the family of generalized complex tori 
\begin{equation*}
J^{p,n}(X/S) = \frac{\mathcal{H}^{2p-n-1}}{\mathcal{F}^{p}+\mathcal{H}^{2p-n-1}_{{\Z}}} 
\end{equation*}
over $S$. 
If $Z_s$ is nullhomologous for every $s\in S$, 
their Abel-Jacobi invariants define the holomorphic section 
\begin{equation*}
\nu_{Z} : S \to J^{p,n}(X/S), \qquad s \mapsto \nu(Z_{s}). 
\end{equation*}
This is called the \textit{normal function} of the cycle family $Z$. 

We consider the \textit{Koszul complex}  
\begin{equation*}
\mathcal{K}^{p,p-n-1} \; : \; 
\mathcal{H}^{p,p-n-1} \stackrel{\nabla}{\longrightarrow} 
\mathcal{H}^{p-1,p-n}\otimes \Omega_{S}^{1} \stackrel{\nabla}{\longrightarrow} 
\mathcal{H}^{p-2,p-n+1}\otimes \Omega_{S}^{2}. 
\end{equation*}
This is a complex of vector bundles on $S$, 
where the homomorphisms $\nabla$ are induced from the Gauss-Manin connection. 
We denote by $H^1\mathcal{K}^{p,p-n-1}$ the cohomology sheaf at the middle. 
For a cycle family $Z$ as above, 
we obtain a holomorphic section of $H^1\mathcal{K}^{p,p-n-1}$ 
by differentiating the normal function $\nu_{Z}$ with respect to the Gauss-Manin connection. 
This is called the \textit{infinitesimal invariant} of $Z$ and denoted by $\delta\nu_{Z}$. 

An important property of the normal function $\nu_{Z}$ is the \textit{admissibility} as explained in \cite{GG}. 
Suppose that we are given a smooth partial compactification $\overline{S}\supset S$ 
whose complement $D=\overline{S}-S$ is a normal crossing divisor. 
We assume that the local system $\mathcal{H}^{2p-n-1}_{{\Z}}$ has 
unipotent local monodromy around $D$. 
Then $(\mathcal{H}^{2p-n-1}, \mathcal{F}^{\bullet})$ has the canonical extension over $\overline{S}$; 
we use the same notation for the extended bundles. 
The Koszul complex $\mathcal{K}^{p,p-n-1}$ over $S$ extends to the logarithmic complex 
\begin{equation*}
\mathcal{K}^{p,p-n-1}(\log D) := (\mathcal{H}^{p-\bullet, p-n-1+\bullet}\otimes \Omega^{\bullet}_{\overline{S}}(\log D), \nabla) 
\end{equation*}
over $\overline{S}$. 
The admissibility of $\nu_{Z}$ in the sense of Green-Griffiths (\cite{GG} Definition 5.2.1) means that 
$\nu_{Z}$ has logarithmic growth around $D$ and satisfies a certain condition on the local monodromy. 
(We will not use the latter.) 
It was noticed in \cite{GG}, \cite{GGK} that 
this is equivalent to the admissibility defined earlier  by M.~Saito \cite{Sa} via mixed Hodge modules. 
When $n=0$, the admissibility of $\nu_{Z}$ is proved in \cite{SZ} Proposition 5.28 and \cite{GGK} Proposition III.B.4. 
The case $n>0$ seems to be more like a folklore: 
see \cite{BPS} p.658 for a brief argument. 

The property that $\nu_{Z}$ has logarithmic growth implies that 
the infinitesimal invariant $\delta\nu_{Z}$ extends to 
a holomorphic section of the vector bundle $H^1\mathcal{K}^{p,p-n-1}(\log D)$ over $\overline{S}$ (see \cite{GG} \S 5.3.1). 
This is the property used in the next \S \ref{sec: main construction}.

\section{The main construction}\label{sec: main construction}

In this section we present and prove the full version of Theorem \ref{thm: intro main construction}. 
In \S \ref{ssec: Koszul} we calculate the relevant Koszul complex as a complex of automorphic vector bundles. 
Our main construction is given in \S \ref{ssec: main construction}. 
The \S \ref{ssec: decomposable} and \S \ref{ssec: rigidity} are complements.

\subsection{Automorphic Koszul complexes}\label{ssec: Koszul}

Let $\Lambda$ be a symplectic lattice of rank $2g$ with $1\leq g \leq 3$. 
Recall that the Hodge bundle ${\E}$ is a sub vector bundle of $\Lambda_{{\C}}\otimes \mathcal{O}_{{\D}}$. 
We also have the trivial local system $\underline{\Lambda}=\Lambda\times {\D}$ inside $\Lambda_{{\C}}\otimes \mathcal{O}_{{\D}}$. 
The quotient 
\begin{equation*}
\mathcal{X}_{{\D}} = \frac{\Lambda_{{\C}}\otimes \mathcal{O}_{{\D}}}{{\E}+\underline{\Lambda}} 
\end{equation*}
is the universal family of abelian varieties over ${\D}$, 
naturally polarized by the symplectic form on $\Lambda$. 
We denote by $\pi \colon \mathcal{X}_{{\D}} \to {\D}$ the projection.  
Then $\mathcal{H}^{1,0}\simeq \pi_{\ast}\Omega_{\pi}^{1} \simeq {\E}$, 
where $\Omega_{\pi}^{1}$ is the relative cotangent bundle of $\pi$. 

More generally, for $(p, q)$ with $p+q\leq g$, we have 
\begin{equation*}
\mathcal{H}^{p,q}\simeq R^{q}\pi_{\ast}\Omega^{p}_{\pi} 
\simeq \wedge^{p}{\E}\otimes \wedge^{q}\mathcal{E}^{\vee}. 
\end{equation*}
This automorphic vector bundle corresponds to the ${\rm GL}(g, {\C})$-representation 
$\wedge^{p}{\St}\otimes \wedge^{q}{\St}^{\vee}$. 
This is reducible: we have the decomposition 
\begin{eqnarray}\label{eqn: Lef decomp GL}
\wedge^{p}{\St}\otimes \wedge^{q}{\St}^{\vee} 
& \simeq & 
V_{\{ 1^p; 1^q \}} \: \oplus \: (\wedge^{p-1}{\St}\otimes \wedge^{q-1}{\St}^{\vee}) \\ 
&\simeq & 
V_{\{ 1^p; 1^q \}} \: \oplus \: V_{\{ 1^{p-1}; 1^{q-1} \}} \: \oplus \: \cdots  \nonumber 
\end{eqnarray}
where 
$\{ 1^p; 1^q \}$ stands for the highest weight $(1^{p}, 0^{g-p-q}, (-1)^{q})$. 
Explicitly, $V_{\{ 1^p; 1^q \}}$ is the kernel of the contraction 
$\wedge^{p}{\St}\otimes \wedge^{q}{\St}^{\vee} \to \wedge^{p-1}{\St}\otimes \wedge^{q-1}{\St}^{\vee}$ 
with the $1$-dimensional trivial summand of ${\St}\otimes {\St}^{\vee}$, 
and $\wedge^{p-1}{\St}\otimes \wedge^{q-1}{\St}^{\vee}$ is embedded in 
$\wedge^{p}{\St}\otimes \wedge^{q}{\St}^{\vee}$ 
by the product with this component. 
The automorphic vector bundle corresponding to $V_{\{ 1^p; 1^q \}}$ 
is the primitive part $\mathcal{H}^{p,q}_{{\prim}}$ of $\mathcal{H}^{p,q}$, 
and \eqref{eqn: Lef decomp GL} corresponds to the Lefschetz decomposition 
\begin{equation}\label{eqn: Lef decomp}
\mathcal{H}^{p,q} \, \simeq \, \mathcal{H}^{p,q}_{{\prim}} \, \oplus \, \mathcal{H}^{p-1,q-1}. 
\end{equation}

The Koszul complex $\mathcal{K}^{p,q}$ for $\mathcal{X}_{{\D}}\to {\D}$, 
regarded as a complex of automorphic vector bundles, corresponds to the complex 
\begin{equation*}
\wedge^{p}{\St}\otimes \wedge^{q}{\St}^{\vee} \to 
\wedge^{p-1}{\St}\otimes \wedge^{q+1}{\St}^{\vee} \otimes {\Sym}^2{\St} \to 
\wedge^{p-2}{\St}\otimes \wedge^{q+2}{\St}^{\vee} \otimes \wedge^2{\Sym}^2{\St} 
\end{equation*}
of ${\rm GL}(g, {\C})$-representations. 
Here the homomorphisms are induced by the natural contractions and products. 
By the Lefschetz decomposition \eqref{eqn: Lef decomp}, 
$\mathcal{K}^{p,q}$ decomposes as 
\begin{equation}\label{eqn: Lef Koszul}
\mathcal{K}^{p,q} = \mathcal{K}^{p,q}_{{\prim}} \oplus \mathcal{K}^{p-1,q-1}, 
\end{equation}
where the primitive part $\mathcal{K}^{p,q}_{{\prim}}$ is  
\begin{equation*}
\mathcal{K}^{p,q}_{{\prim}} \; : \; 
\mathcal{H}^{p,q}_{{\prim}} \to 
\mathcal{H}^{p-1,q+1}_{{\prim}}\otimes \Omega_{{\D}}^1 \to 
\mathcal{H}^{p-2,q+2}_{{\prim}}\otimes \Omega_{{\D}}^2. 
\end{equation*}
This corresponds to the complex  
\begin{equation*}
V_{\{ 1^p; 1^q \}} \to 
V_{\{ 1^{p-1}; 1^{q+1} \}} \otimes {\Sym}^2 \to 
V_{\{ 1^{p-2}; 1^{q+2} \}} \otimes \wedge^2{\Sym}^2 
\end{equation*}
of ${\rm GL}(g, {\C})$-representations, where we abbreviate ${\Sym}^2={\Sym}^2{\St}$. 
By \eqref{eqn: Lef Koszul}, we have 
\begin{equation}\label{eqn: Lef Koszul H1}
H^1\mathcal{K}^{p,q} \, = \, H^1\mathcal{K}^{p,q}_{{\prim}} \oplus H^1\mathcal{K}^{p-1,q-1}. 
\end{equation}

In this paper we will be interested in $\mathcal{K}^{2,g-2}$.  
We denote by $\mathcal{K}_{g}=\mathcal{K}^{2,g-2}_{{\prim}}$ its primitive part. 
Explicitly, $\mathcal{K}_{g}$ is written as 
\begin{eqnarray*}
g=3 & : & \mathcal{H}^{2,1}_{{\prim}} \to 
\mathcal{H}^{1,2}_{{\prim}}\otimes \Omega_{{\D}}^1 \to 
\mathcal{H}^{0,3}\otimes \Omega_{{\D}}^2 \\ 
g=2 & : & \mathcal{H}^{2,0} \to 
\mathcal{H}^{1,1}_{{\prim}}\otimes \Omega_{{\D}}^1 \to 
\mathcal{H}^{0,2}\otimes \Omega_{{\D}}^2 \\ 
g=1 & : & 0 \to \mathcal{H}^{1,0}\otimes \Omega_{{\D}}^1 \to 0  
\end{eqnarray*}

\begin{lemma}\label{lem: H1Kg}
We have $H^1\mathcal{K}_{g}\simeq {\Sym}^4{\E}\otimes {\LL}^{-1}$ for any $1\leq g \leq 3$. 
\end{lemma}

\begin{proof}
When $g=1$, we have 
$\mathcal{H}^{1,0}={\E}={\LL}$ and $\Omega_{{\D}}^1\simeq {\LL}^{\otimes 2}$. 

Next let $g=2$. 
We have $\mathcal{H}^{2,0}= {\LL}$ and $\mathcal{H}^{0,2}\simeq {\LL}^{-1}$.  
Since $\mathcal{H}^{1,1}_{{\rm prim}}$ is the kernel of 
${\E}\otimes {\E}^{\vee}\to \mathcal{O}_{{\D}}$ and 
${\E}^{\vee}\otimes {\LL}\simeq {\E}$, 
we have 
$\mathcal{H}^{1,1}_{{\rm prim}}\otimes {\LL} \simeq {\Sym}^2{\E}$. 
The Koszul complex twisted by ${\LL}$ is written as 
\begin{equation*}
{\LL}^{\otimes 2} \longrightarrow {\Sym}^2{\E}\otimes {\Sym}^2{\E} 
\stackrel{\wedge}{\longrightarrow} \wedge^2{\Sym}^2{\E}. 
\end{equation*}
In the level of representations of ${\rm GL}(2, {\C})$, we have 
\begin{equation*}
{\Sym}^2 \otimes {\Sym}^2 = {\Sym}^4 \oplus {\rm det}^{2} \oplus (\wedge^2{\Sym}^2) 
\end{equation*}
(see \cite{FH} Exercise 6.16 or \S 11.2). 
This shows that the middle cohomology of the above complex is  ${\Sym}^4{\E}$. 

Finally, let $g=3$. 
Then $\mathcal{H}^{1,2}_{{\rm prim}}$ is the kernel of the contraction 
${\E}\otimes \wedge^2{\E}^{\vee}\to {\E}^{\vee}$. 
Since $\wedge^2{\E}^{\vee}\otimes {\LL} \simeq {\E}$ and 
${\E}^{\vee}\otimes {\LL} \simeq \wedge^2{\E}$, 
this contraction twisted by ${\LL}$ is identified with the wedge product ${\E}\otimes {\E}\to \wedge^{2}{\E}$. 
Hence $\mathcal{H}^{1,2}_{{\rm prim}} \otimes {\LL} \simeq {\Sym}^2{\E}$. 
Similarly, we have $\mathcal{H}^{2,1}_{{\rm prim}} \otimes {\LL}^{-1} \simeq {\Sym}^2{\E}^{\vee}$ 
and so $\mathcal{H}^{2,1}_{{\rm prim}} \otimes {\LL} \simeq {\Sym}^{2}(\wedge^{2}{\E})$. 
Thus the Koszul complex twisted by ${\LL}$ can be written as 
\begin{equation*}
{\Sym}^2(\wedge^2{\E}) \longrightarrow {\Sym}^2{\E}\otimes {\Sym}^2{\E} 
\stackrel{\wedge}{\longrightarrow} \wedge^2{\Sym}^2{\E}. 
\end{equation*}
In the level of representations of ${\rm GL}(3, {\C})$, we have 
\begin{equation*}
{\Sym}^2 \otimes {\Sym}^2 = 
{\Sym}^4 \: \oplus \: ({\Sym}^2\!\wedge^2) \: \oplus \: (\wedge^2{\Sym}^2) 
\end{equation*}
(see \cite{FH} Exercise 6.16 or p.189). 
Therefore the middle cohomology of the above complex is ${\Sym}^4{\E}$. 
\end{proof}

Similarly and more easily, 
the non-primitive part $H^1\mathcal{K}^{1, g-3}$ of $H^1\mathcal{K}^{2, g-2}$ is calculated as 
\begin{equation}\label{eqn: non-primitive H1}
H^1\mathcal{K}^{1, g-3} \: \simeq \: 
\begin{cases}
\: \wedge^2{\Sym}^{2}{\E}\otimes {\LL}^{-1}  & \; g=3 \\ 
\: {\Sym}^2{\E} & \; g=2 \\ 
\: 0 & \; g=1 
\end{cases}
\end{equation}

\subsection{The main construction}\label{ssec: main construction}

Let ${\G}$ be a finite-index subgroup of ${\rm Sp}(\Lambda)$ with $-1\not\in {\G}$. 
Over the Siegel modular variety ${\AG}={\D}/{\G}$ 
we have the universal family ${\XG}=\mathcal{X}_{{\D}}/{\G}$ of abelian varieties. 
If we restrict ${\XG}$ over the Zariski open set $\mathcal{A}_{{\G}}^{\circ}\subset {\AG}$ 
where ${\G}$ has trivial stabilizer, 
this is indeed a smooth family of abelian varieties. 

It is often the case that a cycle family can be defined only after taking an etale base change. 
Thus let $f\colon S\to U$ be an etale cover of degree $d$ 
of a Zariski open set $U$ of $\mathcal{A}_{{\G}}^{\circ}$. 
We take the base change $X={\XG}\times_{{\AG}}S$ and 
denote by $\pi \colon X\to S$ the resulting family of abelian varieties. 
From now on, we use subscript for indicating the base of Hodge bundles and Koszul complexes. 
Then $\mathcal{H}_{S}^{p,q} \simeq f^{\ast}\mathcal{H}_{U}^{p,q}$. 
Moreover, since $f$ is etale, we have 
$\mathcal{K}_{S}^{p,q} \simeq f^{\ast}\mathcal{K}_{U}^{p,q}$. 

Let $Z$ be a family of higher Chow cycles of type $(2, 3-g)$ on $X\to S$. 
When $g=3$, we assume that $Z$ is a family of nullhomologous cycles; 
as we explained in \S \ref{ssec: Chow}, this is automatically satisfied when $g\leq 2$. 
The infinitesimal invariant $\delta\nu_{Z}$ of $Z$ is a section of $H^1\mathcal{K}^{2,g-2}_{S}$ over $S$. 
We write $\delta\nu_{Z}^{+}$ for its primitive part according to the decomposition \eqref{eqn: Lef Koszul H1}. 
This is a section of 
\begin{equation*}
H^1\mathcal{K}_{g,S}\simeq f^{\ast}H^1\mathcal{K}_{g,U}\simeq f^{\ast}({\Sym}^4{\E}\otimes {\LL}^{-1}) 
\end{equation*}
over $S$. 

In order to obtain a Siegel modular form, 
we take symmetric polynomials of $\delta\nu_{Z}^{+}$ along the $f$-fibers. 
More precisely, for $1\leq i \leq d$, let $e_i$ be the $i$-th elementary symmetric polynomial in $d$ variables. 
We set 
\begin{equation*}
g_{Z}^{(i)}(u) = e_i(\delta\nu_{Z}^{+}(s_1), \cdots, \delta\nu_{Z}^{+}(s_d)), \qquad u\in U,  
\end{equation*} 
where $f^{-1}(u)= \{ s_1, \cdots, s_d \}$. 
This is well-defined by the symmetricity of $e_{i}$, 
and takes values in the $i$-th symmetric tensor of the fiber of ${\Sym}^4{\E}\otimes{\LL}^{-1}$ over $u$. 
Thus $g_{Z}^{(i)}$ is a holomorphic section of 
\begin{equation*}
{\Sym}^i({\Sym}^4{\E}\otimes{\LL}^{-1}) = {\Sym}^{i}{\Sym}^{4}{\E}\otimes{\LL}^{\otimes -i} 
\end{equation*}
over $U$. 
At this stage, the ${\rm GL}(g, {\C})$-representation ${\Sym}^{i}{\Sym}^{4}$ is reducible unless $g=1$. 
We send $g_{Z}^{(i)}$ by the natural projection 
${\Sym}^i{\Sym}^4 \to {\Sym}^{4i}$ 
and denote the image by $f_{Z}^{(i)}$.

\begin{theorem}\label{thm: main construction}
The section $f_{Z}^{(i)}$ is a meromorphic Siegel modular form 
of weight ${\Sym}^{4i}\otimes {\rm det}^{-i}$ 
which has at most pole of order $i$ along the complement of $U$. 
When $g=1$, $f_{Z}^{(i)}$ is holomorphic at the cusps. 
\end{theorem}

\begin{proof}
By construction, $f_{Z}^{(i)}$ is a holomorphic section of 
${\Sym}^{4i}{\E}\otimes{\LL}^{\otimes -i}$ over the Zariski open set $U$ of ${\AG}$. 
What has to be done is to bound its singularities along the complement of $U$. 

As the first reduction, we shrink ${\G}$. 
We choose a torsion-free subgroup ${\G}'<{\G}$ of finite index and 
let $X'\to S' \to U'$ 
be the pullback of $X\to S \to U$ by 
$\mathcal{A}_{{\G}'}\to {\AG}$, 
where $U'\subset \mathcal{A}_{{\G}'}$ is the inverse image of $U$. 
If $Z'$ is the base change of $Z$ to $X'\to S'$, 
then $f_{Z'}^{(i)}$ is the pullback of $f_{Z}^{(i)}$.  
In other words, $f_{Z'}^{(i)}=f_{Z}^{(i)}$ over ${\D}$. 
Hence our assertion for $f_{Z'}^{(i)}$ implies that for $f_{Z}^{(i)}$. 
Thus, replacing ${\G}$ with ${\G}'$, 
we may assume from the outset that ${\G}$ is torsion-free. 

We may take a smooth partial compactification $S^{+}$ of $S$ such that 
the etale covering $f\colon S\to U$ extends to a (proper) finite morphism 
$f\colon S^{+}\to U^{+}\subset {\AG}$ 
whose image $U^{+}$ has complement of codimension $>1$ in ${\AG}$. 
It suffices to bound the singularities of $f_{Z}^{(i)}$ along $U^{+}-U$. 
Shrinking $S^{+}$ and $U^{+}$ if necessary, 
we may assume that the complements $D_{S}=S^{+}-S$ and $D_{U}=U^{+}-U$ are smooth divisors. 
By the base change from ${\XG}\to {\AG}$, 
the family $X\to S$ extends to a smooth family over $S^{+}$. 
This shows that the relevant variations of Hodge structures naturally extend over $S^{+}$ 
with $\mathcal{H}^{p,q}_{S^{+}}\simeq f^{\ast}\mathcal{H}^{p,q}_{U^{+}}$. 
On the other hand, as for the cotangent bundles,  
$\Omega_{S^{+}}^{k}(\log D_S)$ is isomorphic to $f^{\ast}\Omega_{U^{+}}^{k}(\log D_U)$. 
It follows that 
$\mathcal{K}^{p,q}_{S^{+}}(\log D_S) \simeq f^{\ast}\mathcal{K}^{p,q}_{U^{+}}(\log D_U)$. 
In particular, we have 
\begin{eqnarray}\label{eqn: pullback H^1K}
H^{1}\mathcal{K}_{g,S^{+}}(\log D_S) 
& \simeq & 
f^{\ast} H^{1}\mathcal{K}_{g, U^{+}}(\log D_U) \\ 
& \subset & 
f^{\ast}( (H^{1}\mathcal{K}_{g, U^{+}})(D_U) ) \nonumber  \\ 
& \simeq & 
f^{\ast}({\Sym}^4{\E}\otimes {\LL}^{-1}(D_{U})). \nonumber 
\end{eqnarray} 

Now the admissibility of $\nu_{Z}$ 
tells us that $\delta\nu_{Z}$ extends holomorphically over $S^{+}$ 
as a section of $H^1\mathcal{K}_{S^{+}}^{2,g-2}(\log D_{S})$. 
Hence $\delta\nu_{Z}^{+}$ extends holomorphically over $S^{+}$ 
as a section of $H^1\mathcal{K}_{g, S^{+}}(\log D_{S})$. 
By \eqref{eqn: pullback H^1K} and the properness of $S^{+}\to U^{+}$, 
we find that $f_{Z}^{(i)}$ extends holomorphically over $U^{+}$ 
as a section of ${\Sym}^{4i}{\E}\otimes {\LL}^{\otimes -i}(iD_{U})$. 
This means that, as a section of ${\Sym}^{4i}{\E}\otimes {\LL}^{\otimes -i}$, 
$f_{Z}^{(i)}$ has at most pole of order $i$ along $D_{U}$. 
When $g\geq 2$, this suffices to conclude that $f_{Z}^{(i)}$ is a meromorphic Siegel modular form 
because cusp condition is unnecessary. 

It remains to prove the holomorphicity at the cusps in the case $g=1$. 
We may assume that ${\G}$ is neat. 
We write $\mathcal{A}={\AG}$. 
Let $f\colon \overline{S}\to \mathcal{A}^{\ast}$ be the full compactification of $S^{+}\to U^{+}=\mathcal{A}$, 
where $\overline{S}$ is a smooth projective curve and $\mathcal{A}^{\ast}$ is the compactified modular curve. 
We write $\Delta_{\mathcal{A}}=\mathcal{A}^{\ast}-\mathcal{A}$ for the set of cusps and 
$\Delta_{S}=f^{-1}(\Delta_{\mathcal{A}})=\overline{S}-S^{+}$. 
As is well-known, 
the descent of the local system $\underline{\Lambda}$ over $\mathcal{A}$ has unipotent monodromy around the cusps, 
and the Hodge-theoretic canonical extension of ${\LL}=\mathcal{H}_{\mathcal{A}}^{1,0}$ 
coincides with the extension defining holomorphicity of modular forms at the cusps. 
We use the same notation ${\LL}$ for this extension. 
Moreover, $\Omega^{1}_{\mathcal{A}^{\ast}}(\Delta_{\mathcal{A}})$ is isomorphic to ${\LL}^{\otimes 2}$. 
These show that $R^1\pi_{\ast}{\Z}$ has unipotent monodromy around $\Delta_{S}$, 
the canonical extension of $\mathcal{H}_{S^{+}}^{1,0}$ over $\overline{S}$ is isomorphic to $f^{\ast}{\LL}$, and 
\begin{equation*}
\Omega^{1}_{\overline{S}}(\Delta_{S}+D_{S}) 
\simeq f^{\ast}\Omega^{1}_{\mathcal{A}^{\ast}}(\Delta_{\mathcal{A}}+D_{U}) 
\simeq f^{\ast}{\LL}^{\otimes 2}(D_{U}). 
\end{equation*}
Therefore we have 
\begin{equation*}
H^1\mathcal{K}_{1, \overline{S}}(\log (\Delta_S+D_{S})) \simeq f^{\ast}{\LL}^{\otimes 3}(D_{U}) 
\end{equation*}
over $\overline{S}$. 
The admissibility of $\nu_{Z}$ along $\Delta_{S}$ now implies that 
$\delta\nu_{Z}^{+}=\delta\nu_{Z}$ extends holomorphically over $\Delta_{S}$ as a section of $f^{\ast}{\LL}^{\otimes 3}$. 
(The interior points $D_{S}$ do not affect the argument around $\Delta_{S}$.) 
This implies that $f_{Z}^{(i)}$ is holomorphic at $\Delta_{\mathcal{A}}$ as a section of $\mathcal{L}^{\otimes 3i}$. 
\end{proof}

When $d>1$, it would be often the case that 
$f_{Z}^{(i)}$ vanishes for some $i$ even if $\delta\nu_{Z}^{+}$ does not vanish identically. 
However, at least the top form survives: 

\begin{proposition}
When $\delta\nu_{Z}^{+}\not\equiv 0$, 
we have $f_{Z}^{(d)}\not\equiv 0$. 
\end{proposition}

\begin{proof}
Let $u\in U$. 
Let $F_1, \cdots, F_d$ be the quartic forms corresponding to 
the values of $\delta\nu_{Z}^{+}$ at the $d$ points $f^{-1}(u)$. 
If we send $g_{Z}^{(d)}(u)=e_{d}(F_1, \cdots, F_d)$ by ${\Sym}^d{\Sym}^4\to {\Sym}^{4d}$, 
these quartic forms are multiplied as polynomials 
to yield the degree $4d$ polynomial $F_{1}\cdots F_{d}$.  
Therefore, if $u$ is generic so that  
the $d$ values of $\delta\nu_{Z}^{+}$ at $f^{-1}(u)$ are all nonzero, 
then $f_{Z}^{(d)}$ is nonzero at $u$. 
\end{proof}

Let us comment on some other possible constructions. 

\begin{remark}
When $g\geq 2$, we could also look at other irreducible summands of ${\Sym}^i{\Sym}^4$. 
This yields Siegel modular forms of other weights. 
We decided to work only with the main component ${\Sym}^{4i}$ 
because of its explicit nature and also for the above nonvanishing property. 
\end{remark}

\begin{remark}
When $g\geq 2$, we could also look at the non-primitive part of $\delta\nu_{Z}$. 
By \eqref{eqn: non-primitive H1}, this produces Siegel modular forms of (virtual) weight 
$\wedge^{2}{\Sym}^{2}\otimes \det^{-1}$ when $g=3$, and weight ${\Sym}^2$ when $g=2$. 
The significance of considering the primitive part will be explained in 
\S \ref{ssec: decomposable} and \S \ref{ssec: rigidity}. 
\end{remark} 

\begin{remark}
We may replace ${\XG}\to {\AG}$ with other family of abelian varieties over ${\AG}$ 
which is relatively isogenous to ${\XG}$. 
Indeed, what was essential in Theorem \ref{thm: main construction} 
was the identification of the Hodge bundles and the automorphic vector bundles, 
and this is valid also for other such families. 
\end{remark}

\begin{remark}
When the etale cover $S$ itself has a nice moduli interpretation and 
a satisfactory theory of ``modular forms'' on $S$ is available, 
we could directly study $\delta\nu_{Z}$ itself as such a generalized modular form on $S$ 
without pushing it down to ${\AG}$. 
This is not the approach taken in this paper, 
but what we have in mind is the moduli space $\mathcal{M}_{3}\to \mathcal{A}_{3}$ of genus $3$ curves 
with the theory of Teichm\"uller modular forms (\cite{Ic}, \cite{CFvdG2}). 
See \cite{Hain3} for this direction.   
\end{remark}

\begin{remark}
Another interesting source of normal functions is 
the extension classes of the mixed Hodge structures on the fundamental groups of curves (\cite{Hain}). 
It is known that the first nontrivial extension class corresponds to the Ceresa cycle in the case $g=3$ (\cite{Pu}) 
and to the Collino cycle in the case $g=2$ (\cite{Colombo}). 
The higher extension classes also give Siegel modular forms; 
the admissibility in this case is proved in \cite{Hain2}. 
\end{remark}

\subsection{On the primitive part}\label{ssec: decomposable}

In this subsection we let $g=2, 3$. 
Since Griffiths, the primitive part of normal functions has been used to detect various nontriviality of cycles. 
In our situation, this is expressed in the following form. 

\begin{proposition}\label{prop: decomposable}
Let $Z$ be as in \S \ref{ssec: main construction}. 
Assume that for very general $s\in S$, 
the cycle $Z_{s}$ is algebraically equivalent to zero in the case $g=3$, 
and decomposable in the case $g=2$. 
Then $\delta\nu_{Z}^{+}\equiv 0$. 
\end{proposition}

Here ``very general points'' means 
points in the complement of countably many divisors. 
A $(2, 1)$-cycle on an abelian surface $A$ is called \textit{decomposable} 
if it is in the image of the intersection product 
${\rm Pic}(A)\otimes_{{\Z}} {\C}^{\ast} \to {\rm CH}^2(A, 1)$. 

\begin{proof}
We first consider the case $g=3$. 
For a polarized abelian $3$-fold $A=X_{s}$, 
let $J^{2,0}(A)_{alg}$ be the sub abelian variety of 
the intermediate Jacobian $J^{2,0}(A)$ (usually written as $J^3(A)$) 
corresponding to the maximal sub ${\Q}$-Hodge structure of $H^3(A, {\Q})$ contained in $F^1H^3(A)$. 
The Abel-Jacobi invariant of a $1$-cycle algebraically equivalent to zero takes values in $J^{2,0}(A)_{alg}$. 
On the other hand, we let 
$J^{2,0}(A)_{Lef}$ be the sub abelian variety of $J^{2,0}(A)$ corresponding to the image of the Lefschetz operator
$H^1(A, {\Q})\hookrightarrow H^3(A, {\Q})$. 
This is isogenous to ${\rm Pic}^{0}(A)$. 
Clearly we have 
$J^{2,0}(A)_{Lef} \subset J^{2,0}(A)_{alg}$, 
and $J^{2,0}(A)_{Lef} = J^{2,0}(A)_{alg}$ holds when $A$ is very general. 

Now, by our assumption on $Z$, 
$\nu(Z_{s})$ takes values in $J^{2,0}(A)_{alg} = J^{2,0}(A)_{Lef}$ for very general $s\in S$. 
Since $J^{2,0}(A)_{Lef}$ is stable under deformation (while $J^{2,0}(A)_{alg}$ is not so), 
we see that $\nu(Z_{s})$ takes values in $J^{2,0}(A)_{Lef}$ for every $s\in S$. 
Since $\delta\nu_{Z}^{+}$ is the infinitesimal invariant of the image of $\nu_{Z}$ in the quotient of 
$J^{2,0}(X/S)$ by $J^{2,0}(X/S)_{Lef}$, 
it vanishes identically. 

The case $g=2$ is similar: 
instead of $J^{2,0}(A)_{Lef} \subset J^{2,0}(A)_{alg}$ 
we use the sub tori ${\Z}H\otimes {\C}^{\ast}\subset {\rm NS}(A)\otimes {\C}^{\ast}$ of $J^{2,1}(A)$ 
where $H$ is the polarization. 
The Abel-Jacobi invariant of a decomposable cycle is contained in ${\rm NS}(A)\otimes {\C}^{\ast}$, 
while we have ${\rm NS}(A)={\Z}H$ for very general $A$. 
With this modification, the above argument works in the case $g=2$. 
\end{proof}

Proposition \ref{prop: decomposable} shows that 
passing to the primitive part of $\delta\nu_{Z}$ has the effect of dividing out 
the contribution from families of relatively accessible cycles as above.

\subsection{Rigidity}\label{ssec: rigidity}

In this subsection we let $1\leq g \leq 3$. 
We fix ${\AG}$, but allow $S\to U$ to vary. 
The second significance of considering the primitive part is the following rigidity property. 

\begin{proposition}\label{prop: rigidity}
There are only countably many Siegel modular forms that can be obtained as $f_{Z}^{(i)}$ 
for a cycle family $Z$ as in \S \ref{ssec: main construction}. 
\end{proposition}

This is a consequence of some rigidity results in the theory of algebraic cycles, 
rather than a general property of normal functions (not necessarily of geometric origin). 
We begin with some preliminaries. 
For a polarized abelian $g$-fold $A$, we denote 
\begin{equation*}
J^{2,3-g}(A)_{{\rm prim}} = 
\begin{cases}
J^{2,0}(A)/J^{2,0}(A)_{Lef} & g=3, \\
J^{2,1}(A)/{\Z}H\otimes {\C}^{\ast} & g=2, \\ 
J^{2,2}(A) & g=1, 
\end{cases}
\end{equation*}
where $J^{2,0}(A)_{Lef}$ and $H$ are as in the proof of Proposition \ref{prop: decomposable}. 
For a cycle family $Z$, we denote by $\nu_{Z}^{+}$ 
the image of $\nu_{Z}$ by the projection $J^{2, 3-g}(X/S)\to J^{2, 3-g}(X/S)_{{\rm prim}}$. 
The primitive part $\delta\nu_{Z}^{+}$ of $\delta\nu_{Z}$ is the infinitesimal invariant of $\nu_{Z}^{+}$. 

Let $V$ be a small open set of $\mathcal{A}_{{\G}}^{\circ}$ (in the classical topology). 
By a cycle family over $V$, we mean 
the restriction of a cycle family $Z$ on an etale base-changed family $X\to S$ 
to an open set $\tilde{V}\subset S$ which is projected isomorphically to $V$. 
We let $\mathcal{X}_{V}=\mathcal{X}_{{\D}}|_{V}$ 
be the universal family over $V$. 
Then, via the isomorphism $\tilde{V}\simeq V$, 
$\nu_{Z}^{+}|_{\tilde{V}}$ is regarded as a section of $J^{2,3-g}(\mathcal{X}_{V}/V)_{{\rm prim}}$ over $V$. 

\begin{lemma}\label{lem: rigidity NF}
There are only countably many sections of $J^{2,3-g}(\mathcal{X}_{V}/V)_{{\rm prim}}$ 
that arises as $\nu_{Z}^{+}|_{\tilde{V}}$ for a cycle family $Z$ over $V$. 
\end{lemma}

Proposition \ref{prop: rigidity} follows from Lemma \ref{lem: rigidity NF} 
because $f_{Z}^{(i)}$ is determined by its restriction to $V$, 
which in turn is a combination of $\delta\nu_{Z}^{+}|_{\tilde{V}}$ 
for various $\tilde{V}$ over $V$, 
hence a combination of some sections from a pool of countably many ones. 

\begin{proof}[(Proof of Lemma \ref{lem: rigidity NF})]
We first consider the case $g=1$ which is simplest. 
The Beilinson rigidity (\cite{Be}, see also \cite{MS} Corollary 3.3) 
says that the image of the Abel-Jacobi map 
${\rm CH}^{2}(A, 2)\to J^{2,2}(A)$ 
for an elliptic curve $A$ is countable. 
Hence the normal functions $\nu_{Z}|_{V}$ can take only countably many values at each point of $V$. 
Since a normal function is determined by its values at countably many (Zariski dense) points, 
this shows that there are only countably many sections of the form $\nu_{Z}|_{V}$. 

Next let $g=3$. 
For a polarized abelian $3$-fold $A$, we consider the commutative diagram 
\begin{equation*}
\xymatrix@C+1pc{
{\rm CH}^{2}(A)_{{\rm hom}} \ar[r]^{\nu^{+}} \ar@{->>}[d] & J^{2,0}(A)_{{\rm prim}} \ar@{->>}[d] \\ 
{\rm Griff}^2(A) \ar[r] & J^{2,0}(A)_{tr}  
}
\end{equation*}
where $J^{2,0}(A)_{tr}=J^{2,0}(A)/J^{2,0}(A)_{alg}$, 
${\rm CH}^{2}(A)_{{\rm hom}}$ is the Chow group of nullhomologous $1$-cycles, and 
${\rm Griff}^2(A)$ is their Griffiths group. 
If $A$ is very general, we have $J^{2,0}(A)_{{\rm prim}}=J^{2,0}(A)_{tr}$. 
Since ${\rm Griff}^2(A)$ is a countable group, 
the image of the primitive Abel-Jacobi map $\nu^{+}$ is countable for very general $A$. 
As in the case $g=1$, this implies that  
there are at most countably many sections of the form $\nu^{+}_{Z}|_{V}$. 

The case $g=2$ is similar. 
We replace the above diagram by 
\begin{equation*}
\xymatrix@C+1pc{
{\rm CH}^{2}(A, 1) \ar[r]^{\nu^{+}} \ar@{->>}[d] & J^{2,1}(A)_{{\rm prim}} \ar@{->>}[d] \\ 
{\rm CH}^{2}(A, 1)_{{\rm ind}} \ar[r]^{\nu_{tr}} & J^{2,1}(A)_{tr}  
}
\end{equation*}
where 
${\rm CH}^{2}(A, 1)_{{\rm ind}}= {\rm CH}^{2}(A, 1)/{\rm Pic}(A)\otimes {\C}^{\ast}$ 
is the indecomposable part of ${\rm CH}^{2}(A, 1)$ 
and 
$J^{2,1}(A)_{tr}= J^{2,1}(A)/{\rm NS}(A)\otimes {\C}^{\ast}$.  
Again we have $J^{2,1}(A)_{tr}= J^{2,1}(A)_{{\rm prim}}$ for very general $A$, 
and now the image of $\nu_{tr}$ is countable by \cite{MS} Corollary 3.6. 
\end{proof}

\section{Meromorphic Siegel operator}\label{sec: Siegel operator}

The Siegel operator for holomorphic vector-valued Siegel modular forms is well-established (\cite{We}). 
In this paper we need its extension to the meromorphic case. 
Since we could not find a systematic treatment in the literature, 
we prepare it in this section in a minimally required level. 
This is preliminaries for the next \S \ref{sec: elevator}. 
The \S \ref{ssec: toroidal} and \S \ref{ssec: Siegel auto VB} are recollection of 
the partial toroidal compactification and the holomorphic Siegel operator respectively. 
The meromorphic case is considered in \S \ref{ssec: Siegel meromorphic}. 
We keep the setting of the previous sections, but assume $g>1$ in this section. 

\subsection{Partial toroidal compactification}\label{ssec: toroidal}

Let $I$ be a $1$-dimensional subspace of $\Lambda_{{\Q}}$. 
In this subsection we recall the partial toroidal compactification of ${\AG}$ over the $I$-cusp following \cite{AMRT}. 

Let $P(I)$ be the stabilizer of $I$ in ${\rm Sp}(\Lambda_{{\Q}})$. 
This has the filtration 
\begin{equation*}
U(I) \lhd W(I) \lhd P(I), 
\end{equation*}
where $W(I)$ is the kernel of $P(I)\to {\rm GL}(I)\times {\rm Sp}(I^{\perp}/I)$ (the unipotent radical of $P(I)$) 
and $U(I)$ is the kernel of  $P(I)\to {\rm GL}(I^{\perp})$. 
It turns out that $W(I)$ has the structure of a Heisenberg group with center $U(I)$,  
and we have natural isomorphisms $U(I)\simeq {\Sym}^2I$ and $W(I)/U(I)\simeq (I^{\perp}/I)\otimes I$. 

Naturally associated to $I$ is a $P(I)$-equivariant two-step fibration 
\begin{equation}\label{eqn: Siegel domain}
{\D} \longrightarrow \mathcal{V}_{I} \stackrel{\phi}{\longrightarrow} {\D}_{I}, 
\end{equation}
known as the \textit{Siegel domain realization}, 
where ${\D}\to \mathcal{V}_{I}$ is a fibration of upper half planes 
and $\mathcal{V}_{I}\to {\D}_{I}$ is an affine space bundle over the $I$-cusp ${\D}_{I}$. 
If we choose a section of $\phi$, 
$\mathcal{V}_{I}$ gets isomorphic to the dual of the Hodge bundle on ${\D}_{I}$. 
The group $U(I)$ acts trivially on $\mathcal{V}_{I}$, 
and acts on the fibers of ${\D}\to \mathcal{V}_{I}$ by translation of real parts. 
In the $C^{\infty}$-level, $\mathcal{V}_{I}\to {\D}_{I}$ is a principal $W(I)_{{\R}}/U(I)_{{\R}}$-bundle. 

Given a finite-index subgroup ${\G}$ of ${\rm Sp}(\Lambda)$ with $-1\not\in {\G}$, we write 
\begin{equation*}
{\G}(I) = P(I)\cap {\G}, \qquad 
U(I)_{{\Z}} = U(I)\cap {\G}, \qquad 
\overline{{\G}(I)} = {\G}(I)/U(I)_{{\Z}}. 
\end{equation*}
The quotient ${\D}/U(I)_{{\Z}}$ is a punctured-disc bundle over $\mathcal{V}_{I}$. 
Filling the origins of the punctured-discs, we obtain a disc bundle 
$\overline{{\D}/U(I)_{{\Z}}} \to \mathcal{V}_{I}$. 
The boundary divisor of $\overline{{\D}/U(I)_{{\Z}}}$ 
is the zero section, and is naturally identified with $\mathcal{V}_{I}$. 
We write 
$\Delta_{I} = \mathcal{V}_{I}/\overline{{\G}(I)}$ 
for the boundary divisor of the quotient $\overline{{\D}/U(I)_{{\Z}}}/\overline{{\G}(I)}$. 

By the reduction theory, ${\G}$-equivalence reduces to ${\G}(I)$-equivalence 
in a deleted neighborhood of the boundary of $\overline{{\D}/U(I)_{{\Z}}}$. 
This enables us to glue the Siegel modular variety ${\AG}={\D}/{\G}$ and 
a neighborhood of $\Delta_{I}$ in $\overline{{\D}/U(I)_{{\Z}}}/\overline{{\G}(I)}$ 
along the deleted neighborhood. 
This produces a partial compactification $\mathcal{A}_{{\G}}'$ of ${\AG}$ with boundary divisor $\Delta_{I}$: 
\begin{equation*}
\mathcal{A}_{{\G}}' = {\AG}\sqcup \Delta_{I}.  
\end{equation*}
This $\mathcal{A}_{{\G}}'$ is contained in every full toroidal compactification of ${\AG}$ 
as the inverse image of the Zariski open set ${\AG}\sqcup {\AI}$ of ${\AGast}$. 
%
The boundary projection $\Delta_{I} \to {\AI}$ is the quotient of $\mathcal{V}_{I}\to {\D}_{I}$ by $\overline{{\G}(I)}$. 


\subsection{Siegel operator for automorphic vector bundles}\label{ssec: Siegel auto VB}

In this subsection we recall the Siegel operator for holomorphic Siegel modular forms (\cite{We}, \cite{vdG}).  
For our purpose, we need its reformulation at the level of partial toroidal compactification as given in \cite{Ma2} \S 3.  

Let $\lambda=(\lambda_{1}, \cdots, \lambda_{g})$ be a highest weight for ${\rm GL}(g, {\C})$. 
The vector bundle ${\El}$ descends to a vector bundle on ${\D}/U(I)_{{\Z}}$, again denoted by ${\El}$. 
We choose a maximal isotropic subspace $I_0\subset \Lambda_{{\Q}}$ containing $I$, 
which we fix in what follows. 
Since the $I_0$-trivialization ${\El}\simeq V(I_0)_{\lambda}\otimes \mathcal{O}_{{\D}}$ is 
invariant under $U(I_0)\supset U(I)$, it descends to ${\D}/U(I)_{{\Z}}$. 
We extend ${\El}$ to a vector bundle over $\overline{{\D}/U(I)_{{\Z}}}$, still denoted by ${\El}$, 
so that this isomorphism extends to ${\El}\simeq V(I_0)_{\lambda}\otimes \mathcal{O}_{\overline{{\D}/U(I)_{{\Z}}}}$. 
This is called the \textit{canonical extension} of ${\El}$. 
It is in fact independent of the choice of $I_0$. 

When $\lambda\geq 0$, every holomorphic modular form of weight $\lambda$ 
extends to a holomorphic section of ${\El}$ over $\overline{{\D}/U(I)_{{\Z}}}$. 
We denote by $\mu_{I}(f)$ the vanishing order of $f$ along the boundary divisor $\mathcal{V}_{I}$. 
By construction, this is the same as the vanishing order of $f$ as a $V(I_0)_{\lambda}$-valued function. 
Explicitly, it can be computed as follows. 
Let $f=\sum_{l}a(l)q^{l}$ be the Fourier expansion with respect to $I_{0}$ as in \eqref{eqn: Fourier expansion}. 
Let $v_{I}$ be the positive generator of $U(I)_{{\Z}}\simeq {\Z}$. 
For $l\in U(I_0)_{{\Z}}^{\vee}$, the pairing $(v_{I}, l)\in {\Z}$ measures the vanishing order of the function $q^l$ along $\mathcal{V}_{I}$. 
Therefore we have 
\begin{equation}\label{eqn: mu_I(f)}
\mu_{I}(f) = \min \{ \: (v_{I}, l) \: | \: a(l)\ne 0 \: \}. 
\end{equation}

Let $P(I, I_0)$ be the stabilizer of $I$ in ${\rm GL}(I_0)$,  
and $U(I, I_0)$ be the kernel of the natural map $P(I, I_0) \to {\rm GL}(I)\times {\rm GL}(I_0/I)$.  
Then $U(I, I_0)$ is the unipotent radical of $P(I, I_0)$ and isomorphic to $I\otimes (I_0/I)^{\vee}$. 
We denote by $V(I_0)_{\lambda}^{I}$ the $U(I, I_0)$-invariant part of $V(I_0)_{\lambda}$. 
Then 
\begin{equation}\label{eqn: U-inv decomp}
V(I_0)_{\lambda}^{I} \simeq V(I_0/I)_{\lambda'}\boxtimes (I_{{\C}}^{\vee})^{\otimes \lambda_{g}} 
\end{equation}
as a representation of 
${\rm GL}(I_{0}/I)\times {\rm GL}(I)$ 
where $\lambda'=(\lambda_1, \cdots, \lambda_{g-1})$. 
We let $\mathcal{E}_{\lambda}^{I} \subset {\El}$ be the image of 
$V(I_0)_{\lambda}^{I}\otimes \mathcal{O}_{\overline{{\D}/U(I)_{{\Z}}}}$ 
by the $I_0$-trivialization. 
This sub bundle does not depend on the choice of $I_0$. 
On the other hand, we write ${\E}_{I}$ for the Hodge bundle on ${\D}_{I}$ and ${\LL}_{I}=\det {\E}_{I}$. 

\begin{lemma}
The restriction $\mathcal{E}_{\lambda}^{I}|_{{\VI}}$ is isomorphic to $\phi^{\ast}({\E}_{I})_{\lambda'}$. 
\end{lemma}

\begin{proof}
When $\lambda \geq 0$, this is proved in \cite{Ma2} Proposition 3.5. 
If $\lambda_{g}<0$, we write $\lambda=\lambda^{+}\otimes {\rm det}^{-k}$ with $\lambda^{+}\geq0$. 
Then $V(I_0)_{\lambda}=V(I_0)_{\lambda^{+}}\otimes (\det I_0)_{{\C}}^{\otimes k}$. 
Since $U(I, I_0)$ acts on $\det I_0$ trivially, we have 
$V(I_0)_{\lambda}^{I}=V(I_0)_{\lambda^{+}}^{I}\otimes (\det I_0)_{{\C}}^{\otimes k}$. 
Therefore ${\E}_{\lambda}^{I}={\E}_{\lambda^{+}}^{I}\otimes {\LL}^{\otimes -k}$. 
Since ${\LL}|_{{\VI}}\simeq \phi^{\ast}{\LL}_{I}$, we have  
\begin{equation*}
{\E}_{\lambda}^{I}|_{\mathcal{V}_{I}} 
= ({\E}_{\lambda^{+}}^{I}|_{\mathcal{V}_{I}})\otimes ({\LL}^{\otimes -k}|_{{\VI}}) 
\simeq 
\phi^{\ast}({\E}_{I})_{(\lambda^{+})'}\otimes \phi^{\ast} {\LL}_{I}^{\otimes -k} 
\simeq 
\phi^{\ast}({\E}_{I})_{\lambda'}. 
\end{equation*}  
This proves our assertion. 
\end{proof}

We write $\Phi_{I}{\El}=({\E}_{I})_{\lambda'}$.  
The operation ${\El}\rightsquigarrow \Phi_{I}{\El}$ is composition of restriction and descent. 
This is the Siegel operator for automorphic vector bundles. 
In the level of representations of ${\rm GL}(I_0)\simeq {\rm GL}(g)$, 
this is the process of taking the invariant part for $U(I, I_0)$ and 
regarding it as a representation of ${\rm GL}(I_0/I)\simeq {\rm GL}(g-1)$ 
(i.e., forgetting the twist by $I$ in \eqref{eqn: U-inv decomp}). 

\begin{example}\label{example: Phi Sym det}
We have $\Phi_{I}{\Sym}^{j}{\E}={\Sym}^{j}{\E}_{I}$ and $\Phi_{I}{\LL}^{\otimes k}={\LL}_{I}^{\otimes k}$. 
More generally, we have 
\begin{equation*}
\Phi_{I}({\Sym}^{j}{\E} \otimes {\LL}^{\otimes k}) = {\Sym}^{j}{\E}_{I} \otimes {\LL}_{I}^{\otimes k}. 
\end{equation*}
\end{example}

For a holomorphic modular form $f$ of weight $\lambda \geq 0$, 
the restriction $f|_{{\VI}}$ takes values in $\mathcal{E}_{\lambda}^{I}|_{{\VI}}$ and  
is the pullback of a modular form $\Phi_{I}f$ of weight $\lambda'$ on ${\D}_{I}$, 
i.e., a section of $\Phi_{I}{\El}$ (\cite{Ma2} Proposition 3.6).  
This operation $f\mapsto \Phi_{I}f$, 
composition of restriction and descent, 
is the Siegel operator formulated at the level of toroidal compactification.  
We can pass to a more classical formulation as follows. 
If $f=\sum_{l}a(l)q^l$ is the Fourier expansion of $f$, then 
\begin{equation*}
f|_{{\VI}} = \sum_{l\in U(I)^{\perp}} a(l)q^{l}, \qquad a(l)\in V(I_{0})^{I}_{\lambda} \simeq V(I_{0}/I)_{\lambda'}. 
\end{equation*}
Since $U(I)^{\perp}\cap \mathcal{C}(I_0)$ equals to the boundary cone $\mathcal{C}(I_0/I)$ of $\mathcal{C}(I_0)$, 
the vectors $l$ here are actually contained in $\mathcal{C}(I_0/I)$. 
By the descent from ${\VI}$ to ${\D}_{I}$, 
we obtain the Fourier expansion of $\Phi_{I}f$: 
\begin{equation*}
\Phi_{I}f = \sum_{l\in \mathcal{C}(I_0/I)} a(l)q^{l}, \qquad a(l)\in V(I_{0}/I)_{\lambda'}. 
\end{equation*}
This is the classical formulation of Siegel operator via Fourier expansion. 
  
In the scalar-valued case, when ${\G}$ is neat, 
${\LL}$ extends to an ample line bundle on ${\AGast}$ by the Baily-Borel theory (again denoted by ${\LL}$), 
and we have ${\LL}|_{{\AI}}\simeq {\LL}_{I}$ naturally. 
Then, for a scalar-valued modular form $f$, 
we have $\Phi_{I}f=f|_{{\AI}}$ 
as sections of ${\LL}_{I}^{\otimes k}={\LL}^{\otimes k}|_{{\AI}}$.

\subsection{Meromorphic Siegel operator}\label{ssec: Siegel meromorphic}

We want to extend the formalism of Siegel operator to some \textit{meromorphic} modular forms. 
The definition of the vanishing order $\mu_{I}(f)$ and its expression \eqref{eqn: mu_I(f)} 
are valid even when $f$ is meromorphic. 
Clearly it is necessary that $\mu_{I}(f)\geq 0$ for defining the Siegel operator for $f$, 
but it seems that this is not sufficient (see Remark \ref{remark: Jacobi form}).  
We impose an additional condition as follows. 

Let $\lambda$ be a highest weight for ${\rm GL}(g, {\C})$ (not necessarily $\lambda\geq 0$), 
and $f$ a meromorphic modular form of weight $\lambda$. 
We denote by $P(f)$ the (reduced) image of the pole divisor of $f$ by ${\D}\to {\AG}$. 
This is an algebraic divisor of ${\AG}$. 

\begin{lemma}\label{lem: regularity at cusp}
The following conditions are equivalent. 

(1) The closure of $P(f)$ in ${\AGast}$ does not contain the $I$-cusp ${\AI}$. 

(2) We can write $f=f_1/f_0$ as in Lemma \ref{lem: mero SMF} with $\mu_{I}(f_{0})=0$. 
\end{lemma}

\begin{proof}
(2) $\Rightarrow$ (1): 
If $f=f_1/f_0$, then $P(f)$ is contained in the zero divisor of the holomorphic scalar-valued modular form $f_0$. 
By the theory of holomorphic Siegel operator, 
the intersection of $\overline{{\rm div}(f_0)}$ with ${\AI}$ is the zero divisor of $\Phi_{I}f_{0}\not\equiv 0$. 
Here $\Phi_{I}f_{0}\not\equiv 0$ holds because $\mu_{I}(f_0)=0$. 
Therefore $\overline{{\rm div}(f_0)}$ does not contain ${\AI}$. 
 
(1) $\Rightarrow$ (2): 
We take a point $p\in {\AI}$ not contained in $\overline{P(f)}$. 
We first consider the case ${\G}$ neat. 
Let $\mathcal{I}$ be the ideal sheaf of $\overline{P(f)}$ on ${\AGast}$. 
By the ampleness of ${\LL}$ over ${\AGast}$, for each $\ell>0$, 
the sheaf $\mathcal{I}^{\ell} \otimes {\LL}^{\otimes k}$ is globally generated for $k\gg 0$. 
This implies that we have a section $f_0$ of $\mathcal{I}^{\ell}\otimes {\LL}^{\otimes k}$ over ${\AGast}$ 
such that $f_{0}(p)\ne 0$. 
As a section of ${\LL}^{\otimes k}$, $f_0$ vanishes at $P(f)$ with order $\geq \ell$, 
while it does not vanish identically at ${\AI}$ and so $\mu_{I}(f_0)=0$. 
If we take $\ell$ sufficiently large, the zero of $f_0$ cancels the pole of $f$, 
so $f_1:= f_0\otimes f$ is holomorphic over ${\D}$. 

For general ${\G}$, we take a neat normal subgroup ${\G}'\lhd {\G}$ of finite index. 
Then we do the same argument over $\mathcal{A}_{{\G}'}^{\ast}$, 
where now $f_0$ is first taken to be nonvanishing at every point in the inverse image of $p$ and then 
corrected to be ${\G}/{\G}'$-invariant by taking the average product.  
\end{proof}

\begin{definition}\label{def: regularity at cusp}
When the meromorphic modular form $f$ satisfies the conditions in Lemma \ref{lem: regularity at cusp}, 
we say that $f$ is \textit{regular} at the $I$-cusp. 
\end{definition}

This condition is stronger than $\mu_{I}(f)\geq 0$. 
The latter is equivalent to $\mu_{I}(f_1)\geq \mu_{I}(f_0)$ 
and just means that $f$ as a meromorphic section over $\overline{{\D}/U(I)_{{\Z}}}$ 
is holomorphic at a general point of the boundary divisor ${\VI}$. 

\begin{example}\label{ex: non-regular at cusp}
Let $f_0, f_1$ be scalar-valued holomorphic modular forms with $\mu_{I}(f_0)=\mu_{I}(f_1)>0$. 
If ${\rm div}(f_0)$ has a component not contained in ${\rm div}(f_1)$ but containing ${\AI}$ in its closure, 
then $f=f_1/f_0$ satisfies $\mu_I(f)=0$ but is not regular at $I$ in the sense of Definition \ref{def: regularity at cusp}. 
\end{example}

The condition in Definition \ref{def: regularity at cusp} enables us to define the Siegel operator 
in the style of \S \ref{ssec: Siegel auto VB}: 

\begin{proposition}\label{prop: Siegel meromorphic}
Let $f=\sum_{l}a(l)q^l$ be a meromorphic modular form of weight $\lambda=(\lambda_1, \cdots, \lambda_{g})$ 
which is regular at the $I$-cusp. 
Then the restriction $f|_{{\VI}}$ to ${\VI}$ takes values in ${\E}_{\lambda}^{I}$ 
and is the pullback of a meromorphic modular form $\Phi_{I}f$ 
of weight $\lambda'=(\lambda_1, \cdots, \lambda_{g-1})$ on ${\D}_{I}$. 
If $l\in U(I)^{\perp}$ and $a(l)\ne 0$, then 
$l\in {\Sym}^2(I_0/I)^{\vee}$ and $a(l)\in V(I_0)_{\lambda}^{I}$. 
\end{proposition}

\begin{proof}
We write $f=f_1/f_0$ with $\mu_I(f_{0})=0$ as in Lemma \ref{lem: regularity at cusp} (2). 
Then we have $f_{0}|_{{\VI}}\not\equiv 0$, and hence 
\begin{equation*}
f|_{{\VI}} = (f_{1}|_{{\VI}})/(f_{0}|_{{\VI}}) = \phi^{\ast}\Phi_{I}f_1/\phi^{\ast}\Phi_{I}f_0 
= \phi^{\ast}(\Phi_{I}f_{1}/\Phi_{I}f_{0}). 
\end{equation*}
The first expression shows that 
$f|_{{\VI}}$ takes values in ${\E}_{\lambda}^{I}={\E}_{\lambda^{+}}^{I}\otimes {\LL}^{\otimes -k}$. 
The last expression shows that 
it suffices to set $\Phi_{I}f = \Phi_{I}f_{1}/\Phi_{I}f_{0}$. 
The last assertion follows from the corresponding property of $f_{1}|_{{\VI}}$ and $(f_{0}|_{{\VI}})^{-1}$. 
\end{proof}

We call the operation $f\mapsto \Phi_{I}f$ the \textit{Siegel operator} for the meromorphic modular form $f$. 
Note that if $\mu_{I}(f)=\mu_{I}(f_1)>0$, then $\Phi_{I}f \equiv 0$. 

As in the holomorphic case, the Fourier expansion of $\Phi_{I}f$ is given by 
\begin{equation*}
\Phi_{I}f = \sum_{l\in {\Sym}^2(I_0/I)^{\vee}}a(l)q^l, 
\end{equation*} 
where now the index vectors $l$ run over the intersection of $U(I_0)_{{\Z}}^{\vee}$ 
with a translation of the cone $\mathcal{C}(I_{0}/I)$. 

\begin{remark}\label{remark: Jacobi form}
When $\mu_{I}(f)=0$ but $f$ is not regular at $I$, 
then $f|_{\mathcal{V}_{I}}$ is just a meromorphic Jacobi form of index $0$. 
This no longer descends to ${\D}_{I}$: 
the index vectors in the Fourier expansion are no longer confined to the subspace ${\Sym}^2(I_0/I)^{\vee}$ of $U(I)^{\perp}$. 
\end{remark}

\section{K-theory elevator for rank $1$ degeneration}\label{sec: elevator}

In this section we present the full version of Theorem \ref{thm: degeneration intro} (Theorem \ref{thm: elevator Siegel}). 
In \S \ref{ssec: rank 1} we explain the rank $1$ degeneration of the universal abelian variety. 
In \S \ref{ssec: elevator} we recall the K-theory elevator in the present setting. 
After these preliminaries, we state our result in \S \ref{ssec: main state}. 
Its proof will be given in \S \ref{sec: proof}. 

We keep the setting and notation of \S \ref{sec: Siegel operator}. 
We assume that the image of ${\G}(I)\to {\rm Sp}(I^{\perp}/I)$ does not contain $-1$ 
so that we have the universal abelian variety (rather than Kummer variety) over ${\AI}$ 
with the symplectic lattice $I^{\perp}\cap \Lambda / I\cap \Lambda$. 
This condition also guarantees that the projection 
$\overline{{\D}/U(I)_{\Z}}\to {\AGtor}$ 
is unramified at a general point of the boundary divisor $\mathcal{V}_{I}$ 
(i.e., $I$ is a so-called regular cusp). 
This is satisfied whenever ${\G}$ is neat. 


\subsection{Rank $1$ degeneration}\label{ssec: rank 1}

Several methods are known for extending the universal family ${\XG}\to {\AG}$ 
to a proper family ${\XGtor}\to {\AGtor}$ over the partial toroidal compactification ${\AGtor}$ 
(and furthermore over full toroidal compactifications). 
The first is the method of toroidal compactification with mixed fans, 
which is a generalization of the construction in \cite{AMRT} \S I.4. 
This was worked out in \cite{Na} \cite{HW} for some examples. 
The second method uses the so-called Mumford's construction \cite{Mu}, 
and is worked out in \cite{HKW} \S II.4 in the case $g=2$. 
It seems that over a corank $1$ cusp, both constructions are canonical and produce the same family.  
Anyway, some common properties of ${\XGtor}\to {\AGtor}$ can be summarized as follows. 

For our purpose, it does not matter even if we remove a codimension $>1$ locus from ${\AGtor}$ as necessary. 
We denote by $\pi\colon \mathcal{X}'\to \mathcal{A}'$ the resulting family. 
\begin{enumerate}
\item $\pi\colon \mathcal{X}'\to \mathcal{A}'$ is a projective morphism, 
and $\mathcal{A}'$ is a Zariski open set of ${\AGtor}$ which intersects with the boundary divisor $\Delta_{I}$. 
\item The total space $\mathcal{X}'$ is smooth, and its boundary $\pi^{-1}(\Delta_{I})$ 
is a divisor with normal crossings. 
\item The singular fiber $A=\pi^{-1}(p)$ over a point $p$ of $\Delta_{I}\cap \mathcal{A}'$ is a union 
\begin{equation*}
A = {\proj}_{(1)} \cup {\proj}_{(2)} \cup \cdots \cup {\proj}_{(\alpha)} 
\end{equation*}
of copies of a ${\proj}^1$-bundle ${\proj}(\mathcal{O}_{B}\oplus L)\to B$, 
where $B$ is the $(g-1)$-dimensional abelian variety corresponding to the image of $p$ in ${\AI}$ 
and $L$ is a line bundle on $B$. 
Two adjacent ${\proj}_{(i)}$ and ${\proj}_{(i+1)}$ are 
glued by identifying the $\infty$-section of ${\proj}_{(i)}$ and the $0$-section of ${\proj}_{(i+1)}$. 
We also identify the $\infty$-section of ${\proj}_{(\alpha)}$ and the $0$-section of ${\proj}_{(1)}$ 
with the shift by a point $b\in B$. 
\end{enumerate}
The line bundle $L$ and the shift parameter $b$ are related, 
and this data is parametrized by the fiber of $\Delta_{I} \to {\AI}$. 
In known examples, we have $L\simeq \mathcal{O}_{B}$ if $b=0$. 

The number $\alpha$ of components of $A$ is explicitly determined by $I\cap \Lambda$ and $U(I)_{{\Z}}$. 
For a cycle-theoretic technical reason (cf.~Remark \ref{remark: elevator with several components}), 
we need to add the following condition: 
\begin{enumerate}
\item[(4)] The singular fibers $A$ are irreducible, i.e., $\alpha =1$ in the above notation. 
Thus $A$ is a self-gluing of ${\proj}(\mathcal{O}_{B}\oplus L)$ 
which identifies the $0$-section and the $\infty$-section after a translation. 
\end{enumerate}
See \cite{HW} (and also \cite{HKW} Proposition II.4.5) for examples of such a cusp 
where ${\G}$ is the modular group of $(1, p)$-polarized abelian surfaces with canonical level structure.   

\begin{definition}
We define ${\AItilde}\subset \Delta_{I}$ to be the locus where $L\simeq \mathcal{O}_{B}$ and $b=0$, 
namely the singular fiber $A$ is the self-gluing of $B\times {\proj}^1$ 
which identifies the $0$-section and the $\infty$-section with trivial shift. 
More simply, $A=B\times Q_0$ where $Q_0$ is the irreducible rational curve with one node. 
We call ${\AItilde}$ the \textit{product locus}. 
\end{definition}

In known examples, since $b=0$ implies $L\simeq \mathcal{O}_{B}$, the following holds: 
\begin{enumerate}
\item[(5)] the projection $\phi_{0}\colon {\AItilde}\to {\AI}$ is dominant and has finite fibers. 
\end{enumerate}
We will come back to (1) -- (5) from a general point of view on another occasion. 
For the moment, these are taken as assumptions in the rest of this paper, 
which are satisfied for the examples in \cite{HW}. 
 
Shrinking ${\AItilde}$ if necessary, 
we may assume that the restriction of $\mathcal{X}'\to \mathcal{A}'$ over ${\AItilde}$ 
is isomorphic to $\mathcal{X}_{I}\times Q_0 \to {\AItilde}$ 
where $\mathcal{X}_{I}\to {\AItilde}$ is the pullback of the universal abelian variety over ${\AI}$. 
We fix one such isomorphism. 

\subsection{K-theory elevator}\label{ssec: elevator}

Let $f\colon S'\to U' \subset \mathcal{A}'$ be an etale cover of degree $d$ 
of a Zariski open set $U'$ of $\mathcal{A}'$ which intersects with the product locus ${\AItilde}$. 
Shrinking $U'$ and ${\AItilde}$ if necessary, 
we may assume that ${\AItilde}$ is a closed subset of $U'$. 
We write $S_{I}=f^{-1}({\AItilde})$. 
We denote by $X'\to S'$ the base change of $\mathcal{X}'\to \mathcal{A}'$ by $S'\to U'$. 
Then $X'|_{S_{I}}\simeq X_{I}\times Q_{0}$ where $X_{I}\to S_{I}$ is the base change of 
$\mathcal{X}_{I}\to {\AItilde}$. 
We write $\Delta_{S}\subset S'$ for the inverse image of $\Delta_{U}=U'\cap \Delta_{I}$. 
We also write $U=U'\cap {\AG}$ and 
denote by $X\to S$ the restriction of $X'\to S'$ over $U$. 
The situation is summarized by the commutative diagram 
\begin{equation*}
\xymatrix@C+1pc{
X_{I}\times Q_0 \ar@{^{(}->}[r] \ar[d] & X|_{\Delta_{S}}  \ar@{^{(}->}[r] \ar[d] & X' \ar[d]_{\pi} & X  \ar@{_{(}->}[l] \ar[d] \\ 
S_{I} \ar@{^{(}->}[r] \ar[d] & \Delta_{S} \ar@{^{(}->}[r] \ar[d] & S' \ar[d]_{f} & S \ar@{_{(}->}[l] \ar[d] \\ 
{\AItilde} \ar@{^{(}->}[r] \ar[rd]_{\phi_{0}} & \Delta_{U}  \ar[d]_{\phi} \ar@{^{(}->}[r] & U' & \ar@{_{(}->}[l] U  \\ 
& {\AI} & & 
}
\end{equation*}
Here all squares are base change, 
the middle vertical maps are etale coverings of degree $d$, 
and the upper vertical maps are a family of abelian varieties and its rank $1$ degeneration. 

Now suppose that we have a family $Z'$ of higher Chow cycles of type $(2, 3-g)$ on $X'\to S'$. 
We write $Z$ for the restriction of $Z'$ over $X\to S$. 
On the other hand, if we restrict $Z'$ over $S_{I}$, 
it is a codimension $2$ cycle on $(X'|_{S_{I}})\times {\cube}^{3-g} = X_{I}\times Q_0 \times {\cube}^{3-g}$. 
We pull it back by 
\begin{equation*}
X_{I}\times {\cube}^{4-g} \hookrightarrow X_{I}\times {\proj}^1\times {\cube}^{3-g} \to X_{I}\times Q_0\times {\cube}^{3-g}, 
\end{equation*}
where ${\proj}^1\to Q_{0}$ is the normalization map identifying $0$ and $\infty$. 
By construction, the resulting cycle $Z_{I}$ on $X_{I}\times {\cube}^{4-g}$ satisfies the cocycle construction. 
Hence it gives a family of higher Chow cycles of type $(2, 4-g)$ on $X_{I}\to S_{I}$. 

This type of operation $Z\rightsquigarrow Z_I$ is known as the \textit{K-theory elevator}. 
The first explicit example was discovered by Collino \cite{Co}; 
later a systematic treatment was given in \cite{elevator} \S 2 where this term was coined. 
In what follows, when $g=3$, we assume as before that 
$Z$ is a family of nullhomologous cycles.

\subsection{The statement}\label{ssec: main state}

Recall from Theorem \ref{thm: main construction} that 
the cycle family $Z$ gives a meromorphic Siegel modular form $f_{Z}^{(i)}$ 
of weight ${\Sym}^{4i}\otimes {\rm det^{-i}}$ on ${\AG}$ for each $1\leq i\leq d$. 
In Lemma \ref{lem: holomorphic at boundary}, 
we will prove that $f_{Z}^{(i)}$ is holomorphic at a general point of $\Delta_{I}$. 
However, as explained in \S \ref{ssec: Siegel meromorphic}, 
this is still weaker than the regularity at $I$ in the sense of Definition \ref{def: regularity at cusp}. 
Let us assume that the divisor components of ${\AG}-U$ do not contain ${\AI}$ in their closure. 
Then, by Lemma \ref{lem: regularity at cusp} (1), $f_{Z}^{(i)}$ is regular at $I$. 
This enables us to apply the meromorphic Siegel operator to $f_{Z}^{(i)}$. 
By Proposition \ref{prop: Siegel meromorphic}, 
the resulting $\Phi_{I}f_{Z}^{(i)}$ is a meromorphic Siegel modular form over ${\AI}$, 
again of weight ${\Sym}^{4i}\otimes {\rm det}^{-i}$. 

On the other hand, if consider the cycle family $Z_{I}$ on $X_{I}\to S_{I}$ and 
take the same product construction as 
$\delta\nu_{Z}^{+} \rightsquigarrow f_{Z}^{(i)}$ 
for the etale cover $S_{I}\to {\AItilde}$ of degree $d$ 
(rather than $S_{I}\to {\AI}$ which in general has larger degree), 
we obtain a meromorphic section $f_{Z_{I}}^{(i)}$ 
of $\phi_{0}^{\ast}({\Sym}^{4i}{\E}_{I}\otimes {\LL}_{I}^{\otimes -i})$ over ${\AItilde}$, 
where ${\E}_{I}$ and ${\LL}_{I}$ are the Hodge bundles on ${\AI}$. 

We can now state our main result. 

\begin{theorem}\label{thm: elevator Siegel}
The meromorphic section $f_{Z_{I}}^{(i)}$ over ${\AItilde}$ 
is the pullback of a meromorphic Siegel modular form of weight ${\Sym}^{4i}\otimes {\rm det}^{-i}$ 
over ${\AI}$ (again denoted by $f_{Z_{I}}^{(i)}$), 
and we have 
\begin{equation*}
\Phi_{I}f_{Z}^{(i)} = f_{Z_{I}}^{(i)} 
\end{equation*}
up to constant. 
\end{theorem}

This means that the construction $Z\rightsquigarrow f_{Z}^{(i)}$ is functorial with respect to rank $1$ degeneration, 
where the K-theory elevator corresponds to the Siegel operator. 
The proof of Theorem \ref{thm: elevator Siegel} will be given in \S \ref{sec: proof}.


\begin{remark}\label{remark: elevator with several components}
The K-theory elevator can also be defined when the degenerate abelian variety $A$ has several components (with trivial shift): 
\begin{equation}\label{eqn: several components chain}
A = B\times {\proj}^1 \cup B\times {\proj}^1 \cup \cdots \cup B\times {\proj}^1. 
\end{equation}
Indeed, let $Z_i$ be the restriction of the given cycle on $A$ 
to the $i$-th component $B\times {\cube} \subset B\times {\proj}^1$. 
We take the sum $Z=\sum_{i}Z_i$ regarded as a cycle on a single $B\times {\cube}$. 
Since $Z_i|_{B\times (\infty)} = Z_{i+1}|_{B\times (0)}$, 
we see that $Z|_{B\times (\infty)} = Z|_{B\times (0)}$, 
namely $Z$ is a higher Chow cycle on $B$. 

In this paper we restrict ourselves to the case $A$ irreducible 
because we need to use the limit formula in \cite{elevator} 
which is available only in this setting. 
It would be desirable to extend the limit formula to the reducible case. 
\end{remark}

\section{Siegel operator for Koszul complex}\label{sec: Siegel Koszul}

This section is preliminaries for the proof of Theorem \ref{thm: elevator Siegel}. 
Recall from \S \ref{ssec: Siegel auto VB} that the Siegel operator ${\El} \rightsquigarrow \Phi_{I}{\El}$ 
is defined for automorphic vector bundles. 
In this section we calculate $\Phi_{I}{\El}$ for the components of our Koszul complex $\mathcal{K}_{g}$, 
and compare $\Phi_{I}\mathcal{K}_{g}$ with $\mathcal{K}_{g-1}$. 

\subsection{Cotangent bundle}\label{ssec: Siegel cotangent}

The bundle $\Omega_{{\D}}^{k}$ corresponds to the ${\rm GL}(g, {\C})$-representation $\wedge^{k}{\Sym}^2$.  
This is in general reducible as an automorphic vector bundle, 
but remains irreducible when $k\leq 2$. 
The Siegel operator in this case is as follows. 

\begin{lemma}\label{lem: Siegel cotangent}
Let $k\leq 2$ with $k<g$. 
Then we have $\Phi_{I}\Omega_{{\D}}^{k}=\Omega^{k}_{{\D}_{I}}$. 
\end{lemma}

\begin{proof}
The ${\rm GL}(g, {\C})$-representations ${\Sym}^2$ and $\wedge^{2}{\Sym}^2$ 
are irreducible with highest weights 
$(2, 0^{g-1})$ and $(3, 1, 0^{g-2})$ respectively 
(see \cite{FH} Exercise 6.16 for $\wedge^{2}{\Sym}^2$). 
If we apply $\Phi_{I}$ to these weights, we obtain 
$(2, 0^{g-2})$ and $(3, 1, 0^{g-3})$, 
which correspond to ${\Sym}^2$ and $\wedge^{2}{\Sym}^2$ 
of ${\rm GL}(g-1, {\C})$ respectively. 
This implies $\Phi_{I}\Omega_{{\D}}^{k}=\Omega^{k}_{{\D}_{I}}$. 
\end{proof}

The canonical extension of $\Omega^{k}_{{\D}}$ is 
the logarithmic bundle $\Omega^{k}(\log {\VI})$ of $\overline{{\D}/U(I)_{{\Z}}}$. 
Combining Lemma \ref{lem: Siegel cotangent} with \cite{Ma2} Proposition 4.3, 
we see that $\Omega^{k}(\log {\VI})^{I}|_{{\VI}}$ 
coincides with the sub bundle $\phi^{\ast}\Omega_{{\D}_{I}}^{k}$ of $k$-forms pulled back from ${\D}_{I}$.

\subsection{Hodge bundles}\label{ssec: Siegel Hodge}

We denote by $\mathcal{H}^{p,q}_{I}$, $\mathcal{H}^{p,q}_{I, {\rm prim}}$ 
the (primitive) Hodge bundles on ${\D}_{I}$. 

\begin{lemma}\label{lem: Siegel Hodge}
The following holds. 

(1) $\Phi_{I}\mathcal{H}^{p,0}=\mathcal{H}^{p,0}_{I}$ if $p<g$,  
and $\Phi_{I}\mathcal{H}^{g,0}=\mathcal{H}^{g-1,0}_{I}$. 

(2) $\Phi_{I}\mathcal{H}^{p,q}_{{\rm prim}}=\mathcal{H}^{p,q-1}_{I, {\rm prim}}$ if $q>0$. 
\end{lemma}

\begin{proof}
(1) The bundle $\mathcal{H}^{p,0}$ corresponds to the highest weight $(1^{p}, 0^{g-p})$. 
If we apply $\Phi_{I}$, we obtain $(1^{p}, 0^{g-p-1})$ if $p<g$, and $(1^{g-1})$ if $p=g$. 
These correspond to $\mathcal{H}^{p,0}_{I}$ and $\mathcal{H}^{g-1,0}_{I}$ respectively. 

(2) The bundle $\mathcal{H}^{p,q}_{{\rm prim}}$ corresponds to the highest weight 
\begin{equation*}
\{ 1^{p}; 1^{q} \} = (1^p, 0^{g-p-q}, (-1)^{q}). 
\end{equation*}
When $q>0$, the Siegel operator sends it to 
\begin{equation*}
(1^p, 0^{g-p-q}, (-1)^{q-1}) = \{ 1^{p}; 1^{q-1} \}, 
\end{equation*}
which corresponds to $\mathcal{H}^{p,q-1}_{I, {\rm prim}}$. 
\end{proof}

We explain a geometric interpretation of the isomorphisms in Lemma \ref{lem: Siegel Hodge} on the product locus ${\AItilde}$. 
We begin with (1). 
Let $\Omega^{p}_{\pi}(\log )$ be the bundle of relative logarithmic $p$-forms for $\pi\colon \mathcal{X}'\to \mathcal{A}'$. 
The canonical extension of $\mathcal{H}^{p,0}$ is naturally isomorphic to $\pi_{\ast}\Omega^{p}_{\pi}(\log )$. 
Let $A=B\times Q_0$ be the degenerate abelian variety over a point of ${\AItilde}$. 
The space $H^{0}(A, \Omega^{p}_{\pi}(\log )|_{A})$ 
consists of logarithmic $p$-forms on $B\times {\proj}^1$ 
which is holomorphic outside $B\times (0)$ and $B\times (\infty)$ 
and has opposite residues there. 
This sits in the exact sequence 
\begin{equation}\label{eqn: log form A}
0 \to H^0(\Omega^{p}_{B}) \to H^{0}(A, \Omega^{p}_{\pi}(\log )|_{A}) \to H^0(\Omega^{p-1}_{B}) \to 0, 
\end{equation}
where the first map is the pullback by $B\times {\proj}^1\to B$ 
and the second map takes the residue at $B\times (0)$. 
This sequence can be splitted as 
\begin{equation}\label{eqn: split log form}
H^{0}(A, \Omega^{p}_{\pi}(\log )|_{A}) = H^0(\Omega^{p}_{B}) \; \oplus \; (H^0(\Omega^{p-1}_{B})  \wedge  dz/z), 
\end{equation}
where $z$ is the coordinates on ${\proj}^1$. 

By the $I_0$-trivialization $\mathcal{H}^{p,0}\simeq \wedge^{p}I_{0}^{\vee}\otimes \mathcal{O}_{{\D}}$, 
the exact sequence \eqref{eqn: log form A} corresponds to 
\begin{equation*}\label{eqn: p-form I_0}
0 \to \wedge^{p}(I_0/I)^{\vee} \to \wedge^{p}I_{0}^{\vee} \to I^{\vee}\otimes \wedge^{p-1}(I_0/I)^{\vee} \to 0. 
\end{equation*}
In particular, the residue map 
$H^{0}(A, \Omega^{p}_{\pi}(\log )|_{A}) \to H^0(\Omega^{p-1}_{B})$ 
is identified with the map 
$\wedge^{p}I_{0}^{\vee} \to \wedge^{p-1}(I_0/I)^{\vee}$ 
given by the contraction with a generator of $I\cap \Lambda$. 
When $p<g$, the $U(I, I_0)$-invariant part of $\wedge^{p}I_{0}^{\vee}$ is the subspace $\wedge^{p}(I_0/I)^{\vee}$. 
Thus the fiber of $(\mathcal{H}^{p,0})^{I}$ over $[A]\in {\AItilde}$ 
is identified with the subspace $H^0(\Omega^{p}_{B})$ of $H^{0}(A, \Omega^{p}_{\pi}(\log )|_{A})$. 
On the other hand, when $p=g$, the subspace $H^0(\Omega_{B}^{g})=\wedge^{g}(I_{0}/I)^{\vee}$ vanishes for degree reason. 
In this case, the isomorphism $(\mathcal{H}^{g,0})^{I}\simeq \mathcal{H}^{g-1,0}_{I}$ 
corresponds instead to the second map in \eqref{eqn: log form A}. 

Next we consider (2) of Lemma \ref{lem: Siegel Hodge}. 
We write 
\begin{equation*}
H^{p,q}(A) = H^{0}(A, \Omega^{p}_{\pi}(\log )|_{A}) \otimes \overline{H^{0}(A, \Omega^{q}_{\pi}(\log )|_{A})}. 
\end{equation*}
This is identified with the fiber of $\mathcal{H}^{p,q}$ over $[A]$. 

\begin{lemma}\label{prop: Siegel Hodge geometric II}
Let $q>0$. 
The fiber of $(\mathcal{H}^{p,q}_{{\rm prim}})^{I}$ over $[A]\in {\AItilde}$ 
corresponds to the subspace 
$H^{p,q-1}_{{\rm prim}}(B)\wedge (d\bar{z}/\bar{z})$ of $H^{p,q}(A)$. 
\end{lemma}

\begin{proof}
By the $I_0$-trivialization, 
$H^{p,q}(A)$ corresponds to $\wedge^{p}I_{0}^{\vee}\otimes \wedge^{q}I_{0}$. 
For simplicity we take a splitting $I_0=I\oplus (I_0/I)$ which is dual to \eqref{eqn: split log form}. 
This induces the splitting 
\begin{eqnarray*}
\wedge^{p}I_{0}^{\vee}\otimes \wedge^{q}I_{0} 
& = & 
\wedge^{p}(I_0/I)^{\vee}\otimes \wedge^{q}(I_0/I) \\ 
& & 
\oplus \; 
(\wedge^{p-1}(I_0/I)^{\vee}\wedge I^{\vee} )\otimes \wedge^{q}(I_0/I) \\ 
& & 
\oplus \; 
\wedge^{p}(I_0/I)^{\vee}\otimes (\wedge^{q-1}(I_0/I) \wedge I) \\ 
& & 
\oplus \; 
(\wedge^{p-1}(I_0/I)^{\vee}\wedge I^{\vee} ) \otimes (\wedge^{q-1}(I_0/I) \wedge I). 
\end{eqnarray*}
The Siegel operator singles out the third term 
(the lowest weight part for ${\rm GL}(I^{\vee})\simeq {\C}^{\ast}$). 
This corresponds to 
$H^{p,q-1}(B)\wedge (d\bar{z}/\bar{z})$. 
Taking the primitive part, we obtain 
$H^{p,q-1}_{{\rm prim}}(B)\wedge (d\bar{z}/\bar{z})$. 
\end{proof}

By the above calculation, the isomorphism 
$(\mathcal{H}^{p,q}_{{\rm prim}})^{I}\simeq \mathcal{H}^{p,q-1}_{I, {\rm prim}}$ 
over $[A]\in {\AItilde}$ is identified with the residue map 
\begin{equation*}
H^{p,q-1}_{{\rm prim}}(B) \wedge (d\bar{z}/\bar{z}) \to H^{p,q-1}_{{\rm prim}}(B), \qquad 
\omega\wedge \, d\bar{z}/\bar{z} \mapsto \omega. 
\end{equation*}

\subsection{Koszul complex}\label{ssec: Siegel Koszul}

Let $\mathcal{K}_{g}$ be the Koszul complex over ${\D}$ defined in \S \ref{ssec: Koszul}. 
By Lemma \ref{lem: H1Kg} and Example \ref{example: Phi Sym det}, 
we have $\Phi_{I}(H^1\mathcal{K}_{g})=H^1\mathcal{K}_{g-1}$. 
Let us look at the Siegel operator for the Koszul complex itself, not just its $H^1$. 
We denote by $\Phi_{I}\mathcal{K}_{g}$ the complex in genus $g-1$ 
obtained by applying the Siegel operator to each term of $\mathcal{K}_{g}$. 

\begin{proposition}\label{prop: Siegel 3 to 2}
We have $\Phi_{I}\mathcal{K}_{3}=\mathcal{K}_{2}$. 
Moreover, this coincides with the complex 
\begin{equation}\label{eqn: Siegel 3 to 2}
\Phi_{I}\mathcal{H}^{2,1}_{{\rm prim}} \to 
\Phi_{I}\mathcal{H}^{1,2}_{{\rm prim}} \otimes \Phi_{I}\Omega_{{\D}}^1 \to 
\Phi_{I}\mathcal{H}^{0,3} \otimes \Phi_{I}\Omega_{{\D}}^2,  
\end{equation}
namely $\Phi_{I}$ and $\otimes$ commute. 
\end{proposition}

\begin{proof}
The coincidence of $\mathcal{K}_{2}$ with \eqref{eqn: Siegel 3 to 2} follows from 
Lemma \ref{lem: Siegel cotangent} and Lemma \ref{lem: Siegel Hodge}.  
Since $\Phi_{I}$ commutes with the twist by ${\LL}=\mathcal{H}^{3,0}$, we find that 
$\Phi_{I}(\mathcal{H}^{0,3} \otimes \Omega_{{\D}}^2)$ equals to 
$\mathcal{H}^{0,2}_{I} \otimes \Omega_{{\D}_{I}}^2$. 
It remains to verify 
\begin{equation*}
\Phi_{I}(\mathcal{H}^{1,2}_{{\rm prim}} \otimes \Omega_{{\D}}^1) = \mathcal{H}^{1,1}_{I, {\rm prim}} \otimes \Omega_{{\D}_{I}}^1. 
\end{equation*} 
We take the twist by ${\LL}$. 
Both $\mathcal{H}^{1,2}_{{\rm prim}}\otimes {\LL}$ and $\Omega_{{\D}}^1$ 
correspond to the ${\rm GL}(3, {\C})$-representation ${\Sym}^2$. 
If we apply $\Phi_{I}$ to the ${\rm GL}(3, {\C})$-representation 
\begin{equation*}
{\Sym}^2\otimes {\Sym}^2 = 
{\Sym}^4 \oplus ({\Sym}^2 \! \wedge^2 ) \oplus (\wedge^2{\Sym}^2), 
\end{equation*}
we obtain the ${\rm GL}(2, {\C})$-representation 
\begin{equation*}
{\Sym}^4 \oplus {\rm det}^{\otimes 2} \oplus (\wedge^2{\Sym}^2) = 
{\Sym}^2\otimes {\Sym}^2. 
\end{equation*}
This corresponds to $\mathcal{H}^{1,1}_{I, {\rm prim}} \otimes {\LL}_{I} \otimes \Omega_{{\D}_{I}}^1$. 
\end{proof}

In the next genus $g=2$, we have $\Phi_{I} \mathcal{K}_{2}\ne \mathcal{K}_{1}$. 
Instead the following holds. 

\begin{proposition}\label{prop: Sigel Koszul g=2}
The complex $\mathcal{K}_{1}$ equals to 
$0\to \Phi_{I}\mathcal{H}^{1,1}_{{\rm prim}} \otimes \Phi_{I} \Omega_{{\D}}^1 \to 0$. 
\end{proposition} 

\begin{proof}
By Lemma \ref{lem: Siegel cotangent} and Lemma \ref{lem: Siegel Hodge}, we have 
$\Phi_{I} \mathcal{H}^{1,1}_{{\rm prim}} = \mathcal{H}^{1,0}_{I} = {\LL}_{I}$ and 
$\Phi_{I} \Omega_{{\D}}^1 = \Omega_{{\D}_{I}}^1$. 
\end{proof}

In both cases, $H^1\mathcal{K}_{g-1}=\Phi_{I}(H^1\mathcal{K}_{g})$ 
is naturally realized as a sub automorphic vector bundle of 
$\Phi_{I} \mathcal{H}^{1,g-1}_{{\rm prim}} \otimes \Phi_{I} \Omega_{{\D}}^1$. 
This implies that $(H^1\mathcal{K}_{g})^{I}$ is realized as a sub bundle of 
$(\mathcal{H}^{1,g-1}_{{\rm prim}})^{I} \otimes (\Omega_{{\D}}^1)^{I}$. 
This is the property used in the next \S \ref{sec: proof}.

\section{Proof of Theorem \ref{thm: elevator Siegel}}\label{sec: proof}

In this section we prove Theorem \ref{thm: elevator Siegel}. 
We keep the setting and notation of \S \ref{sec: elevator}.

\subsection{Extension to the boundary}\label{ssec: cusp reduction cover}

Recall from \S \ref{ssec: main construction} that 
$\delta\nu_{Z}^{+}$ is a holomorphic section of $f^{\ast}({\Sym}^4{\E}\otimes {\LL}^{-1})$ over $S$. 
We regard it as a meromorphic section of $f^{\ast}({\Sym}^4{\E}\otimes {\LL}^{-1})$ over $S'$,  
where 
${\Sym}^4{\E}\otimes {\LL}^{-1}$ extends over $U'\subset \mathcal{A}_{{\G}}'$ by the canonical extension. 

\begin{lemma}\label{lem: holomorphic at boundary}
As a section of $f^{\ast}({\Sym}^4{\E}\otimes {\LL}^{-1})$, 
$\delta\nu_{Z}^{+}$ extends holomorphically over the boundary divisor $\Delta_{S}$. 
\end{lemma}

\begin{proof}
When $g=1$, this is proved in the proof of Theorem \ref{thm: main construction}. 
The case $g>1$ is similar, and the argument proceeds as follows. 
The logarithmic extension $\mathcal{K}_{g, U'}(\log \Delta_{U})$ 
of the Koszul complex $\mathcal{K}_{g, U}$ over $U'$ 
coincides with the canonical extension as a complex of automorphic vector bundles. 
It follows that $H^{1}\mathcal{K}_{g, U'}(\log \Delta_{U})$ coincides with the canonical extension of 
$H^1\mathcal{K}_{g} \simeq {\Sym}^4{\E}\otimes {\LL}^{-1}$ as an automorphic vector bundle. 
Then, by the admissibility of $\nu_{Z}$ along $\Delta_{S}$, we see that $\delta\nu_{Z}^{+}$ as a section of 
\begin{equation*}
H^1\mathcal{K}_{g, S'}(\log \Delta_{S}) \simeq f^{\ast}H^1\mathcal{K}_{g, U'}(\log \Delta_{U}) 
\simeq  f^{\ast}({\Sym}^4{\E}\otimes {\LL}^{-1}) 
\end{equation*}
extends holomorphically over $S'$. 
\end{proof}

Next we show that a property in the Siegel operator (Proposition \ref{prop: Siegel meromorphic}) 
holds already for $\delta\nu_{Z}^{+}$ before pushing it down to ${\AG}$. 

\begin{lemma}\label{lem: cusp reduction cover}
The restriction of $\delta\nu_{Z}^{+}$ to $\Delta_{S}$ takes values in the sub vector bundle 
$f^{\ast}({\Sym}^{4}{\E^{I}}\otimes {\LL}^{-1})$ of $f^{\ast}({\Sym}^{4}{\E}\otimes {\LL}^{-1})$. 
\end{lemma}

\begin{proof}
Since the meromorphic Siegel modular form $f_{Z}^{(d)}$ is regular at $I$ by our assumption, 
it takes values in the sub bundle $({\Sym}^{4d}{\E})^{I}\otimes {\LL}^{\otimes -d}$ at $\Delta_{U}$ 
by Proposition \ref{prop: Siegel meromorphic}. 
This is the product of $\delta\nu_{Z}^{+}|_{\Delta_{S}}$ along the fibers of the etale covering 
$f|_{\Delta_{S}} \colon \Delta_{S}\to \Delta_{U}$. 
We take a basis $x_1, \cdots, x_{g}$ of $I_{0}^{\vee}$ such that 
$(I_0/I)^{\vee}$ is spanned by $x_1, \cdots, x_{g-1}$. 
Then the fiber of ${\Sym}^{\ell}{\E}$ over a point $p$ of $\Delta_{U}$ is 
identified with the space of homogeneous polynomials of $x_1, \cdots, x_{g}$ of degree $\ell$, 
and that of $({\Sym}^{\ell}{\E})^{I}={\Sym}^{\ell}({\E}^{I})$ consists of polynomials of $x_1, \cdots, x_{g-1}$. 

Now let $P_1, \cdots, P_d$ be the quartic forms corresponding to the values of $\delta\nu_{Z}^{+}$ 
at the $d$ points $f^{-1}(p)$. 
(We ignore the twist by ${\LL}$.) 
The property of $f_{Z}^{(d)}$ stated above means that 
the product $P_1 \cdots P_d$ is a polynomial of $x_1, \cdots, x_{g-1}$, 
i.e., has degree $0$ with respect to $x_{g}$. 
This implies that each $P_i$ is already a polynomial of $x_1, \cdots, x_{g-1}$. 
This in turn means that $\delta\nu_{Z}^{+}$ takes values in $f^{\ast}({\Sym}^{4}{\E}^{I}\otimes {\LL}^{-1})$ 
at every point of $f^{-1}(p)$. 
\end{proof}

\subsection{The limit formula}\label{ssec: limit formula}

In this subsection we recall the limit formula of \cite{elevator} \S 2.5 in the present setting. 
Let $X\to T$ be a $1$-parameter degeneration of abelian varieties of dimension $g$, 
where $T$ is a small disc, 
such that the total space $X$ is smooth and 
the central fiber is of the form $B\times Q_0$ 
with $B$ a smooth abelian variety of dimension $g-1$. 
Let $Z$ be a family of nullhomologous higher Chow cycles of type $(2, 3-g)$ on $X/\Delta$. 
For $t\ne 0$ the Abel-Jacobi invariant $\nu(Z_{t})$ of $Z_t$ takes value in 
\begin{equation*}
\frac{H^g(X_{t}, {\C})}{F^2+H^g(X_{t}, {\Z})} \simeq \frac{F^{g-1}H^g(X_{t}, {\C})^{\vee}}{H_g(X_{t}, {\Z})}. 
\end{equation*}
For $t=0$, let $Z_0$ be the K-theory elevator of $Z$. 
Recall that this is a higher Chow cycle of type $(2, 4-g)$ on $B$. 
Its Abel-Jacobi invariant $\nu(Z_{0})$ takes value in 
\begin{equation*}
\frac{H^{g-1}(B, {\C})}{F^2+H^{g-1}(B, {\Z})} \simeq \frac{F^{g-2}H^{g-1}(B, {\C})^{\vee}}{H_{g-1}(B, {\Z})}. 
\end{equation*}

Now we take a $C^{\infty}$-family of test differential forms $\omega_{t} \in F^{g-1}H^g(X_{t})$ 
which converges as $t\to 0$ to a differential form on $B\times Q_0$ of the form 
$(2\pi i)^{-1} \eta_{0}\wedge (dz/z)$ 
where $\eta_{0}\in F^{g-2}H^{g-1}(B)$. 
Then Theorem 2.2 of \cite{elevator} says that 
\begin{equation}\label{eqn: limit formula}
\lim_{t\to 0} \langle \nu(Z_{t}), \omega_{t} \rangle = \langle \nu(Z_{0}), \eta_0 \rangle. 
\end{equation}

\begin{remark}\label{remark: limit formula}
(1) Strictly speaking, the paper \cite{elevator} works with cycles with rational coefficients, 
i.e., modulo torsion cycles, 
but it seems that the calculation in \cite{elevator} \S 2.5 is valid for cycles with integral coefficients. 
We may also work with rational coefficients, 
as this does not affect the infinitesimal invariant. 

(2) What is actually proved in \cite{elevator} \S 2.5 is equality of the KLM currents. 
Thus, if we interpret $\langle \nu(Z_{t}), \, \cdot \, \rangle$ as the KLM current given by 
\cite{elevator} Equation (2.4), 
it can be paired with closed $g$-forms not necessarily in $F^{g-1}$, 
and \eqref{eqn: limit formula} is valid for such a family $(\omega_{t})$. 
The KLM current represents the Abel-Jacobi invariant when tested against forms in $F^{g-1}$. 
\end{remark}

\subsection{The case $g=2$}\label{ssec: g=2}

Now we prove Theorem \ref{thm: elevator Siegel}, first in the case $g=2$. 
The bulk of proof is to establish the equality over $S_{I}$. 
Note that $\delta\nu_{Z_{I}}^{+}=\delta\nu_{Z_{I}}$ in this case. 
By Lemma \ref{lem: cusp reduction cover}, 
both $\delta\nu_{Z}^{+}|_{S_{I}}$ and $\delta\nu_{Z_{I}}^{+}$ are sections of 
the vector bundle 
\begin{equation*}\label{eqn: VB over SI}
f^{\ast}({\Sym}^4{\E}^{I}\otimes {\LL}^{-1})|_{S_{I}} 
= f^{\ast}\phi^{\ast}({\Sym}^4{\E}_{I}\otimes {\LL}_{I}^{-1})|_{S_{I}} 
= (f|_{S_{I}})^{\ast}\phi_{0}^{\ast}({\Sym}^{4}{\E}_{I}\otimes {\LL}_{I}^{-1}) 
\end{equation*}
over $S_{I}$. 

\begin{proposition}\label{prop: equality SI}
We have $\delta\nu_{Z}^{+}|_{S_{I}} = \delta\nu_{Z_{I}}^{+}$ up to constant 
as sections of the above vector bundle. 
\end{proposition}

\begin{proof}
Since the problem is local, we may work in a neighborhood of a general point of $S_{I}$. 
Since $S'$ is etale over $U'$, 
we may pass to a general point $p$ of ${\AItilde}$, 
and assume that the cycle family is defined on a neighborhood of $p$ in $\mathcal{A}_{{\G}}'$. 

Let $(\tau, w, t)$ be the local coordinates around $p$ induced by the Siegel domain realization \eqref{eqn: Siegel domain}, 
where $\tau$ is coordinates on the cusp ${\D}_{I}$ (the upper half plane), 
$w$ is fiber coordinates of ${\VI}\to {\D}_{I}$, and 
$t$ is fiber coordinates of $\overline{{\D}/U(I)_{{\Z}}}\to {\VI}$. 
In particular, the boundary divisor $\Delta_{I}$ is defined by $t=0$. 
Let $(\tau_{0}, w_{0}, 0)$ be the coordinates of $p$. 
By a change of variable with $w$ if necessary, 
we may assume that ${\AItilde}\subset \Delta_{I}$ is locally defined by $w=w_0$ around $p$, 
i.e., locally the curve $\tau\mapsto (\tau, w_{0}, 0)$. 
It suffices to prove 
\begin{equation*}
\lim_{t\to 0}\delta\nu_{Z}^{+}(\tau_{0}, w_{0}, t) = \delta\nu_{Z_{I}}(\tau_{0}, w_{0}, 0). 
\end{equation*}
%

By Proposition \ref{prop: Sigel Koszul g=2},  
$\Phi_{I}H^1\mathcal{K}_{2}$ is identified with   
$\Phi_{I}\mathcal{H}^{1,1}_{{\rm prim}}\otimes \Phi_{I}\Omega_{{\D}}^{1}$ 
and hence 
\begin{equation*}
(H^1\mathcal{K}_{2})^{I} = (\mathcal{H}^{1,1}_{{\rm prim}})^{I} \otimes \Omega^{1}(\log \Delta_{I})^{I}.
\end{equation*} 
By Lemma \ref{lem: Siegel cotangent}, 
$\Omega^{1}(\log \Delta_{I})^{I}$ is generated by $d\tau$ and hence 
the vector field $\partial_{\tau}=\frac{\partial}{\partial \tau}$ is a dual section of $\Omega^{1}(\log \Delta_{I})^{I}$. 
By Lemma \ref{prop: Siegel Hodge geometric II}, 
the fiber of $(\mathcal{H}^{1,1}_{{\rm prim}})^{I}$ over $p$ is identified with 
$H^{1,0}(B)\wedge (d\bar{z}/\bar{z})$ 
where $B\times Q_0$ is the degenerate abelian surface over $p$. 
We take a dual $C^{\infty}$-section $\omega = (\omega_{t})_{t}$ of $(\mathcal{H}^{1,1}_{{\rm prim}})^{I}$. 
When $t\ne 0$, $\omega_{t}$ is represented by a primitive $(1, 1)$-form on the abelian surface $X_t$, 
while $\omega_0$ is represented by $(2\pi i)^{-1}\eta_{0}\wedge (dz/z)$ 
with $\eta_{0}$ a closed $1$-form on the elliptic curve $B$. 
Since we know that 
both $\lim_{t\to 0} \delta\nu_{Z}^{+}(\tau_0, w_0, t)$ and $\delta\nu_{Z_{I}}$ take value in $(H^1\mathcal{K}_{2})^{I}$, 
it suffices to prove the equality after taking the pairing with $\partial_{\tau}\otimes \omega$ as above. 

Since the isomorphism $(\mathcal{H}^{1,1}_{{\rm prim}})^{I}|_{\Delta_{I}}\simeq \phi^{\ast}\mathcal{H}^{1,0}_{I}$ 
is given by getting rid of $d\bar{z}/\bar{z}$ from $H^{1,0}(B)\wedge (d\bar{z}/\bar{z})$ 
and since $d\bar{z}/\bar{z}$ is dual to $(2\pi i)^{-1}dz/z$ up to constant, 
it suffices to prove  
\begin{equation}\label{eqn: equality to be proved}
\lim_{t\to 0} \, \langle \delta\nu_{Z}(\tau_{0}, w_{0}, t), \partial_{\tau}\otimes \omega_{t} \rangle 
= \langle \delta\nu_{Z_{I}}(\tau_{0}, w_{0}, 0), \partial_{\tau}\otimes \eta_{0} \rangle. 
\end{equation} 
The left hand side is equal to 
\begin{eqnarray*}
& & \lim_{t\to 0} \, \langle (\nabla_{\partial_{\tau}}\nu_{Z})(\tau_{0}, w_{0}, t), \omega_{t} \rangle \\ 
& = &  
\lim_{t\to 0} \, \{ \partial_{\tau}\langle \nu_{Z}, \omega_{t} \rangle (\tau_{0}, w_{0}, t) 
- \langle \nu_{Z}(\tau_0, w_0, t), \nabla_{\partial_{\tau}}\omega_{t} \rangle \} \\ 
& = & 
\partial_{\tau} \lim_{t\to 0} \, \langle \nu_{Z}, \omega_{t} \rangle (\tau_{0}, w_{0}, t) 
- \lim_{t\to 0} \langle \nu_{Z}(\tau_0, w_0, t), \nabla_{\partial_{\tau}}\omega_{t} \rangle. 
\end{eqnarray*}
By \eqref{eqn: limit formula}, this is equal to 
\begin{equation}\label{eqn: equality to be proved II}
\partial_{\tau} \langle \nu_{Z_{I}}, \eta_{0} \rangle (\tau_{0}, w_{0}, 0) - 
\langle \nu_{Z_{I}}(\tau_0, w_0, 0), \nabla_{\partial_{\tau}}\eta_{0} \rangle. 
\end{equation} 
Here, as for the second term, 
we used Remark \ref{remark: limit formula} (2) and the equality 
\begin{eqnarray*}
\lim_{t\to 0} \nabla_{\partial_{\tau}} \omega_{t} 
& = & 
\nabla_{\partial_{\tau}} \lim_{t\to 0}\omega_{t}  
 =  
(2\pi i)^{-1}\nabla_{\partial_{\tau}}(\eta_{0}\wedge dz/z) \\ 
& = & 
(2\pi i)^{-1}(\nabla_{\partial_{\tau}}\eta_{0})\wedge dz/z.  
\end{eqnarray*}
Finally, \eqref{eqn: equality to be proved II} is equal to 
\begin{equation*}
\langle \nabla_{\partial_{\tau}}\nu_{Z_{I}}(\tau_{0}, w_{0}, 0), \eta_{0} \rangle  =  
\langle \delta\nu_{Z_{I}}(\tau_{0}, w_{0}, 0), \partial_{\tau}\otimes \eta_{0} \rangle. 
\end{equation*}
This proves \eqref{eqn: equality to be proved}. 
\end{proof} 

A few argument remains for completing the proof of Theorem \ref{thm: elevator Siegel}.  
For $1\leq i \leq d$, let $f_{Z_{I}}^{(i)}$ be the $i$-th relative symmetric polynomial of $\delta\nu_{Z_{I}}^{+}$ 
taken along the fiber of $S_{I}\to {\AItilde}$, 
composed with ${\Sym}^{i}{\Sym}^{4}\to {\Sym}^{4i}$. 
This is a section of 
$\phi_{0}^{\ast}({\Sym}^{4i}{\E}_{I}\otimes {\LL}_{I}^{\otimes -i})$ 
over ${\AItilde}$. 
By Proposition \ref{prop: equality SI}, we have 
\begin{equation}\label{eqn: equality AItilde}
f_{Z_{I}}^{(i)} = f_{Z}^{(i)}|_{{\AItilde}} 
\end{equation}
as an equality over ${\AItilde}$. 
Since $f_{Z}^{(i)}$ is a meromorphic modular form, 
we see from Proposition \ref{prop: Siegel meromorphic} that 
$f_{Z}^{(i)}|_{{\AItilde}}$ descends from ${\AItilde}$ to ${\AI}$. 
Then the equality \eqref{eqn: equality AItilde} guarantees that 
$f_{Z_{I}}^{(i)}$ descends from ${\AItilde}$ to ${\AI}$, 
the result again denoted by $f_{Z_{I}}^{(i)}$. 
By descending the equality \eqref{eqn: equality AItilde} from ${\AItilde}$ to ${\AI}$, 
we obtain 
$\Phi_{I}f_{Z}^{(i)} = f_{Z_{I}}^{(i)}$ 
over ${\AI}$. 
This proves Theorem \ref{thm: elevator Siegel} in the case $g=2$.

\subsection{The case $g=3$}\label{ssec: g=3}

The proof of Theorem \ref{thm: elevator Siegel} in the case $g=3$ is similar to the case $g=2$. 
In this case 
$(H^1\mathcal{K}_{3})^{I}$ no longer coincides with  
$(\mathcal{H}^{1,2}_{{\rm prim}})^{I} \otimes \Omega^{1}(\log \Delta_{I})^{I}$,  
but is still realized as its sub vector bundle  
by Proposition \ref{prop: Siegel 3 to 2}. 
The cusp ${\D}_{I}$ is now the $3$-dimensional Siegel upper half space, 
so we use $3$-variable coordinates $(\tau_{ij})_{i,j}$ instead of a single variable $\tau$. 
Then we replace $\partial_{\tau}\otimes \omega$ by 
$\partial_{\tau_{ij}}\otimes \omega$ as the test dual sections where now 
$\omega_{t}\in H^{2,1}_{{\rm prim}}(X_t)$ 
converging to $(2\pi i)^{-1}\eta_{0}\wedge (dz/z)$ 
with $\eta_{0}\in H^{1,1}_{{\rm prim}}(B)$.  
After this modification, the proof of Proposition \ref{prop: equality SI} is valid in the case $g=3$. 
The step deducing Theorem \ref{thm: elevator Siegel} from Proposition \ref{prop: equality SI} 
remains the same.



\begin{thebibliography}{99}

\bibitem{AMRT}
Ash, A.; Mumford, D.; Rapoport, M.; Tai, Y. 
\textit{Smooth compactifications of locally symmetric varieties.} 2nd ed. 
Cambridge. Univ. Press, 2010. 


\bibitem{Be}Beilinson, A.~A. 
\textit{Higher regulators and values of L-functions.}  
J. Soviet Math. \textbf{30} (1985) 2036--2070. 

\bibitem{Bl1}Bloch, S. 
\textit{Algebraic cycles and higher K-theory.}
Adv. in Math. \textbf{61} (1986), no.3, 267--304.

\bibitem{Bl2}Bloch, S.
\textit{Algebraic cycles and the Beilinson conjectures.} 
The Lefschetz centennial conference, Part I (Mexico City, 1984), 65--79.
Contemp. Math. \textbf{58},  AMS, 1986. 


\bibitem{BPS}Brosnan, P.; Pearlstein, G.; Schnell, C.
\textit{The locus of Hodge classes in an admissible variation of mixed Hodge structure.} 
C. R. Math. Acad. Sci. Paris \textbf{348} (2010), no.11-12, 657--660.

\bibitem{CFvdG1}Cl\'ery, F.; Faber, C.; van der Geer, G.
\textit{Covariants of binary sextics and vector-valued Siegel modular forms of genus two.} 
Math. Ann. \textbf{369} (2017), no.3-4, 1649--1669.

\bibitem{CFvdG2}Cl\'ery, F.; Faber, C.; van der Geer, G.
\textit{Concomitants of ternary quartics and vector-valued Siegel and Teichm\"uller modular forms of genus three.} 
Selecta Math. \textbf{26} (2020), no.4, Paper No. 55, 39 pp.

\bibitem{Co}Collino, A.
\textit{Griffiths' infinitesimal invariant and higher K-theory on hyperelliptic Jacobians.}
J. Algebraic Geom. \textbf{6} (1997), no.3, 393--415.

\bibitem{CP}Collino, A.; Pirola, G.~P. 
\textit{The Griffiths infinitesimal invariant for a curve in its Jacobian.}
Duke Math. J. \textbf{78} (1995), no.1, 59--88.

\bibitem{Colombo}Colombo, E. 
\textit{The mixed Hodge structure on the fundamental group of hyperelliptic curves and higher cycles.} 
J. Algebraic Geom. \textbf{11} (2002), no.4, 761--790. 

\bibitem{elevator}del \'Angel R.~P.~L.; Doran, C.; Kerr, M.; Lewis, J.; Iyer, J.; M\"uller-Stach, S.; Patel, D. 
\textit{Specialization of cycles and the K-theory elevator.}
Comm. Number Theory Phys. \textbf{13} (2019), no.2, 299--349.

\bibitem{FH}Fulton, W.; Harris, J. 
\textit{Representation theory.} 
GTM \textbf{129}, Springer, 1991. 

\bibitem{vdG}van der Geer, G. 
\textit{Siegel modular forms and their applications.} 
The 1-2-3 of modular forms, 181--245, Springer, 2008. 

\bibitem{GG}Green, M.; Griffiths, P. 
\textit{Algebraic cycles and singularities of normal functions.} 
Algebraic cycles and motives. Vol. 1, 206--263.
LMS Lect. Note Ser., \textbf{343}, 
Cambridge Univ. Press, 2007

\bibitem{GGK}Green, M.; Griffiths, P. ; Kerr, M. 
\textit{N\'eron models and limits of Abel-Jacobi mappings.}
Compos. Math. \textbf{146} (2010), no.2, 288--366.

\bibitem{Hain}Hain, R.~M. 
\textit{The geometry of the mixed Hodge structure on the fundamental group.} 
Algebraic geometry, Bowdoin, 1985, 247--282. 
Proc. Symp. Pure Math., \textbf{46}, Part 2, AMS, 1987 

\bibitem{Hain2}Hain, R.~M. 
\textit{The de Rham homotopy theory of complex algebraic varieties. II.} 
K-Theory \textbf{1} (1987), no.5, 481--497.

\bibitem{Hain3}Hain, R.~M. 
\textit{The rank of the normal functions of the Ceresa and Gross--Schoen cycles.} 
arXiv:2408.07809

\bibitem{HW}Hulek, K.; Weintraub, S.~H. 
\textit{The principal degenerations of abelian surfaces and their polarisations.}
Math. Ann. \textbf{286} (1990), no.1-3, 281--307.

\bibitem{HKW}Hulek, K.; Kahn, C.; Weintraub, S.~H. 
\textit{Moduli spaces of abelian surfaces: compactification, degenerations, and theta functions.}
De Gruyter, 1993. 

\bibitem{Ic}Ichikawa, T. 
\textit{Teichm\"uller modular forms of degree 3.}
Amer. J. Math. \textbf{117} (1995), no.4, 1057--1061.

\bibitem{KLM}Kerr, M.; Lewis, J.; M\"uller-Stach, S. 
\textit{The Abel-Jacobi map for higher Chow groups.}
Compos. Math. \textbf{142} (2006), no.2, 374--396.


\bibitem{Ma1}Ma, S. 
\textit{Vector-valued orthogonal modular forms.}
to appear in Mem. Eur. Math. Soc., arXiv: 2209.10135

\bibitem{Ma2}Ma, S. 
 \textit{Siegel operator for holomorphic differential forms.} 
arXiv: 2409.04315

\bibitem{MS}M\"uller-Stach, S.~J. 
\textit{Constructing indecomposable motivic cohomology classes on algebraic surfaces.} 
J. Algebraic Geom. \textbf{6} (1997), no.3, 513--543. 

\bibitem{Mu}Mumford, D. 
\textit{An analytic construction of degenerating abelian varieties over complete rings.} 
Compositio Math. \textbf{24} (1972), 239--272. 

\bibitem{Na}Namikawa, Y. 
\textit{A new compactification of the Siegel space and degeneration of Abelian varieties. II.}
Math. Ann. \textbf{221} (1976), no.3, 201--241.

\bibitem{No}Nori, M. 
\textit{Algebraic cycles and Hodge-theoretic connectivity.} 
Invent. Math. \textbf{111} (1993), 349--373. 



\bibitem{Pu}Pulte, M.~J. 
\textit{The fundamental group of a Riemann surface: mixed Hodge structures and algebraic cycles.} 
Duke Math. J. \textbf{57} (1988), no.3, 721--760. 
 
\bibitem{Sa}Saito, M. 
\textit{Admissible normal functions.}
J. Algebraic Geom. \textbf{5} (1996), no.2, 235--276.

\bibitem{SZ}Steenbrink, J.~H.~M.; Zucker, S. 
\textit{Variation of mixed Hodge structure. I.} 
Invent. Math. \textbf{80} (1985), 489--542. 

\bibitem{We}Weissauer, R. 
\textit{Vektorwertige Siegelsche Modulformen kleinen Gewichtes.}
J. Reine Angew. Math. \textbf{343} (1983), 184--202.

\end{thebibliography}
\end{document}